\documentclass[ams,12pt]{dgpaper}

\usepackage{fullpage}

\begin{document}

\title{V{\~u} Ng{\d{o}}c's Conjecture on focus-focus singular fibers with multiple pinched points}
\author{\'Alvaro Pelayo \quad Xiudi Tang}
\date{}

\begin{abstract}
  We classify, up to fiberwise symplectomorphisms, a saturated neighborhood of a singular fiber of an integrable system (which is proper onto its image and has connected fibers) containing $k > 1$ focus-focus critical points.
  Our result shows that there is a one-to-one correspondence between such neighborhoods and $k$ formal power series, up to a $(\Z_2 \times D_k)$-action, where $D_k$ is the $k$-th dihedral group.
  The $k$ formal power series determine the dynamical behavior of the Hamiltonian vector fields associated to the components of the momentum map on the symplectic manifold $(M,\omega)$ near the singular fiber containing the $k$ focus-focus critical points.
  This proves a conjecture of San V{\~u} Ng{\d{o}}c from 2002.
\end{abstract}

\maketitle

\section{Introduction} \label{sec:introduction}

Our goal in this paper is to provide a classification of saturated neighborhoods of a compact connected fiber $\mathcal{F}$ containing only non-degenerate focus-focus singular points of integrable systems $F = M \to \R^2$ on symplectic $4$-manifolds up to isomorphism. 
Here a point $m \in \mathcal{F}$ being of \emph{focus-focus type} means that there is an $E \in GL(2, \R)$ such that if $(J, H) = E \circ (F - F(m))$ then the Hessians $\mathcal{H}_J(m)$ of $J$ and $\mathcal{H}_H(m)$ of $H$ are simultaneously symplectically conjugate to the following matrices
\begin{align*}
  \mathcal{H}_J(m) &\sim \begin{pmatrix}0&0&0&1\\0&0&-1&0\\0&-1&0&0\\1&0&0&0\end{pmatrix}, &\mathcal{H}_H(m) &\sim \begin{pmatrix}0&1&0&0\\1&0&0&0\\0&0&0&1\\0&0&1&0\end{pmatrix}.
\end{align*}
Eliasson's normal form theorem (see Eliasson~\cite{Eliasson_1984} and V{\~u} Ng{\d{o}}c--Wacheux~\cite{MR3098203}) states that the aforementioned normal form is not only achieved linearly but also nonlinearly, in the sense that any focus-focus singular point of $F$ has a neighborhood such that $(M, \omega, F)$ restricted to it is isomorphic to $(\R^4, \omega_0, q)$ restricted to a neighborhood of the origin, where
\begin{equation*}
  q(x_1, y_1, x_2, y_2) = (x_1 y_2 - x_2 y_1, x_1 y_1 + x_2 y_2).
\end{equation*}
In particular, focus-focus singular points are isolated (and $\mathcal{F}$ is allowed to have any finite number of such singular points).
Throughout this paper, with very few exceptions that we point out explicitly, we assume that $F$ \emph{is proper onto its image with connected fibers}.
If all singular points on a compact connected fiber $\mathcal{F}$ have focus-focus type then there are finitely many such points, say $k \in \N$ is the number which is the only topological invariant of such fibers, and then $\mathcal{F}$ is homeomorphic to a torus pinched $k$ times (Zung~\cite{MR1454718, MR1908416}).

In 2003, V{\~u} Ng{\d{o}}c proved~\cite{MR1941440} that integrable systems at a compact connected fiber with one non-degenerate singular point of focus-focus type is classified (up to a $(\Z_2 \times \Z_2)$-action if one does not specify directions, as clarified in~\cite{MR3765971}) by a formal power series $\mathsf{s} \in \mathsf{R}_{2\pi X}$, where $\mathsf{R}$ is the space of formal power series in two variables $X, Y$, without the constant term, and $\mathsf{R}_{2\pi X} \defeq \mathsf{R} / (2\pi X)\Z$.
In his paper, he also stated a claim for the case $k > 1$, which already appeared in the arXiv version of the paper in 2002, and in 2003~\cite[Section~7]{MR1941440} with a sketch of the argument.

He claimed that a neighborhood of a compact connected fiber with precisely $k \in \N$ non-degenerate critical points of focus-focus type is classified up to isomorphisms by $k$ formal power series in $\mathsf{R}$, with $(k-1)$ of which measure the obstruction to construct a semiglobal momentum map in the Eliasson normal form simultaneously at two different singular points, and the other of which is the Taylor series of the action integral in a neighborhood of the critical fiber, vanishing at the origin, desingularized at each singular point.
Essentially, this turns out to be the case, as we verify in the current paper.
Indeed, let $\mathsf{R}_+$ be the group of formal power series in $\mathsf{R}$ with the coefficients of the $Y$ term positive, and let $\Z_k = \Set{\Zkzero, \Zkone, \dotsc, \overline{k-1}}$ and let $D_k$ stand for the $k$-th Dihedral group of order $2k$.
Our main theorem, which is more technical so we state later once we have introduced the necessary ingredients for its precise formulation, will essentially say that in a small symplectic neighborhood of a focus-focus fiber with $k$ singular points of focus-focus type, an integrable system on a symplectic $4$-manifold is classified up to isomorphisms, by $k$ formal power series 
\begin{equation*}
  \mathsf{s_{\Zkzero}} \in \mathsf{R}_{2\pi X}, \mathsf{g}_{\Zkzero, \Zkone}, \mathsf{g}_{\Zkone, \overline{2}}, \dotsc, \mathsf{g}_{\overline{k-2}, \overline{k-1}} \in \mathsf{R}_+
\end{equation*}
modulo an effective $(\Z_2 \times D_k)$-action. 
Here the $D_k$-action controls the way the focus-focus points are ordered and $\Z_2$-action reflects the direction of the natural $\mathbb{S}^1$-action.

\begin{remark} \label{rem:focus-semitoric-almost-toric}
  Focus-focus singular points appear naturally as singularities of semitoric \cite{MR2534101,MR2784664} and almost-toric systems \cite{MR3843843,MR2024634}.
  In fact, any semitoric or almost-toric integrable system satisfies the assumptions above.
  An explicit example of a semitoric system which includes a twice pinched torus for certain values of the parameters is given in \cite{HohPal}.
\end{remark}

This formulation however hides the fact that we give a step-by-step explicit construction of the invariants $\mathsf{s_{\Zkzero}} \in \mathsf{R}_{2\pi X}$ and $\mathsf{g}_{\Zkone, \overline{2}}, \dotsc, \mathsf{g}_{\overline{k-2}, \overline{k-1}} \in \mathsf{R}_+$; for the case of $k=1$, the only invariant is $\mathsf{s_{\Zkzero}}$, and this has been computed (at least some of its terms have been computed) by several authors for important cases such as the coupled spin-oscillator system~\cite{MR3924558,MR2864789}, the spherical pendulum~\cite{MR3017036}, and the coupled angular-momenta~\cite{MR4039778,MR3927110}.
Roughly speaking, $\mathsf{s_{\Zkzero}}$ measures the global singular behavior of the Hamiltonian vector fields $\mathcal{X}_{f_1}$ and $\mathcal{X}_{f_2}$, where $F = (f_1, f_2)$, near the fiber containing the focus-focus singular point.
The travel times of the flows of these vector fields exhibit a singular behavior, of logarithmic type, as they approach the singular points.
The remaining $(k-1)$ Taylor series $\mathsf{g}_{j, j+\Zkone}$ account for the difference between the Eliasson normal forms at the singular points $m_j$ and $m_{j+\Zkone}$.

We will attempt to make the proof as self-contained as possible, and to facilitate this we include a section on preliminares with a quick review of the ingredients we need for the proof of the aforementioned classification result.
We would like to point out a related recent work by Bosinov--Izosimov~\cite{MR4057723} where the authors give a \emph{smooth} classification of semiglobal germs at compact focus-focus leaves.
The smooth invariants they define can be computed from the symplectic invariants of the present paper, see \cref{rem:compare-Izosimov}.
We refer to the recent article \cite{MR4200677} for an introduction and review of recent progress on the symplectic geometry of integrable systems.

\section{Preliminaries} \label{sec:preliminary}

The goal of this section is twofold.
First, we briefly review the basic terminology and results which we need to state the main theorem of the paper in the following sections; this is done with the goal of making the paper as self-contained as possible.
Second, once we have set up the basic notions, we derive some consequences and some variations of them which we will need to prove the main theorem in the following section; most of the new content of this section is concentrated on \cref{ssec:normal-form} about automorphisms of local normal forms, where we prove technical results about local symplectomorphisms and their possible extensions.

\subsection{Basic review of integrable systems} \label{ssec:integrable-system}

In this subsection, we recall the definition of integrable systems and semitoric integrable systems and discuss some of the concepts and their properties that are related to finding the semiglobal symplectic invariants.

Let $(M, \omega)$ be a $2n$-dimensional symplectic manifold.
For a smooth map $f \colon M \to \R$ we denote by $\mathcal{X}_f = -\omega^{-1}(\der f) \in \mathfrak{X}(M)$ the \emph{Hamiltonian vector field} of $f$.
For any smooth maps $f, g \colon M \to \R$ we define their \emph{Poisson bracket} $\Set{f, g} = \omega(\mathcal{X}_f, \mathcal{X}_g)$ and say they are \emph{Poisson commutative} if their Poisson bracket vanishes.

\begin{definition} \label{def:integrable-system}
  Let $(M, \omega)$ be a $2n$-dimensional symplectic manifold.
  Let 
  \begin{equation*}
    F = (f_1, \dotsc, f_n) \colon M \to \R^n
  \end{equation*}
  be a smooth map such that $f_1, \dotsc, f_n$ are functionally independent (i.e. $\der f_1, \dotsc, \der f_n$ are linearly independent) almost everywhere and pairwise Poisson commutative.
  In this case we call $F$ a \emph{momentum map} on $M$ and $(M, \omega, F)$ an \emph{integrable system}.
  Let $\IntSys$ be the collection of all integrable systems.
\end{definition}

It is worth noting that the fiber of a momentum map $F$ is Lagrangian near any regular point of $F$ it goes through.
For any regular point $x \in M$ of $F$, let $b = F(x)$.
Since $\mathcal{X}_1(x), \dotsc, \mathcal{X}_n(x)$ are linearly independent and $\dim T_xF^{-1}(b) = n$, they span $T_xF^{-1}(b)$.
But then $\{f_i, f_j\}(x) = \omega(\mathcal{X}_i, \mathcal{X}_j)(x) = 0$ implies that $T_xF^{-1}(b) \subset T_xM$ is Lagrangian.
Therefore, integrable systems are sometimes considered as singular Lagrangian fibrations.

Let $B = F(M)$ and let $U \subset B$ be an open subset.
For each $\beta \in \Omega^1(U)$, its \emph{Hamiltonian vector field} $\mathcal{X}_\beta = -\omega^{-1} (F^* \beta)$ is a vector field on $F^{-1}(U)$.
When exists, let $\Psi_\beta \colon F^{-1}(U) \to F^{-1}(U)$ be the time-$1$ map of the flow of $\mathcal{X}_\beta$.
Throughout this paper we adopt the following definition for smooth functions or forms on subsets $B \subset \R^n$: let $X$ be a smooth manifold.
A map $f \colon B \to X$ is \emph{smooth} if any $b \in B$ has an open neighborhood $U$ in $\R^n$ and a smooth function $\wt{f} \colon U \to X$ which coincides with $f$ in $B \cap U$.

\begin{definition} \label{def:integrable-system-complete}
  An integrable system $(M, \omega, F)$ is called \emph{complete} if $\Psi_\beta$ exists for any open subset $U \subset B$ and each $\beta \in \Omega^1(U)$.
  Equivalently, the flow of $\mathcal{X}_\beta$ exists for all time.
\end{definition}

\begin{remark} \label{rem:complete-compact=fiber}
  The completeness in \cref{def:integrable-system-complete} is automatic when the moment map has compact fibers, or more specifically, is proper onto its image.
\end{remark}

Let $(M, \omega, F)$ be a complete integrable system.
For any open subset $U \subset B = F(M)$ let $\Period^{(M, \omega, F)}(U) = \Set{\beta \in \Omega^1(U) \mmid \Psi_{2\pi \beta} = \identity}$.
We will use $\Period$ omitting the superscripts if there is no ambiguity.

\begin{definition} \label{def:period-sheaf}
  We call $\Period$ the \emph{period sheaf} of $(M, \omega, F)$, which is a sheaf of abelian groups over $B$.
  The local sections of $\Period$ are \emph{period forms}.
\end{definition}

Consider the presheaf $U \mapsto \Omega^1(U) / 2\pi \Period(U)$, denoted by $\Omega^1 / 2\pi \Period$ of abelian groups on $B$.
On one hand, any global section $\tau \in (\Omega^1 / 2\pi \Period)(B)$ yields a symplectomorphism $\Psi_{\tau}$ of $(M, \omega)$ as a fiberwise translation by the representatives of stalks of $\tau$.
On the other hand, let $P, Q, R \colon B \to M$ be smooth sections of $F$.
If every $b \in B$ has an open neighborhood $U$ in $B$ such that
\begin{equation*}
  \tau^{PQ}_U = \Set{\beta \in \Omega^1(U) \mmid \Psi_\beta \circ \Res{P}_U = \Res{Q}_U}
\end{equation*}
is nonempty, then $\tau^{PQ}_U$ is a coset of $2\pi \Period(U)$ in $\Omega^1(U)$, whose germ at $b$ is an element of the stalk of $\Omega^1 / 2\pi \Period$ at $b$.
Then those $\tau^{PQ}_U$, for $b$ ranging in $B$ and $U$ being neighborhoods as above, glue to a global section $\tau^{PQ} \in (\Omega^1 / 2\pi \Period)(B)$.
In this case we call $\tau^{PQ}$ the \emph{translation form} from $P$ to $Q$.
The translation forms satisfy the additivity property $\tau^{PQ} + \tau^{QR} = \tau^{PR}$, whenever the two forms on the left-hand side are defined.

\begin{definition} \label{def:vertically-transitive}
  A complete integrable system $(M, \omega, F)$ is \emph{vertically transitive} if $\tau^{PQ}$ is nonempty for any smooth sections $P, Q \colon B \to M$ where $B = F(M)$.
\end{definition}

Here we specify the morphisms of integrable systems we consider, and then give some properties of such morphisms.

\begin{definition} \label{def:morphism}
  Let $(M, \omega, F)$, $(M', \omega', F')$ be integrable systems and set $B = F(M)$, $B' = F'(M')$.
  If $G \colon B \to B'$ and $\varphi \colon M \to M'$ satisfy that $F' \circ \varphi = G \circ F$, then we say that $\varphi$ \emph{lifts} $G$.
  A \emph{morphism} from $(M, \omega, F)$ to $(M', \omega', F')$ is a pair $(\varphi \colon M \to M', G \colon B \to B')$ of smooth maps with $\varphi$ lifting $G$ such that $\varphi^* \omega' = \omega$.
  A morphism $(\varphi, G)$ is said to be an \emph{isomorphism} if $\varphi$ is a diffeomorphism, and hence so is $G$.
\end{definition}

\begin{remark} \label{rem:morephism-semitoric}
  The concept of an isomorphism of semitoric systems \cite{MR2534101} is more restrictive than \cref{def:morphism}, which is similar to an isomorphism of integrable systems preserving the direction as in \cref{def:direction,def:moduli-space}.
\end{remark}

\begin{lemma} \label{lem:symplecto-auto}
  Let $(M, \omega, F)$ be a complete integrable system, let $U \subset B = F(M)$ be an open subset and let $\tau \in \Omega^1(U)$.
  Then $\Psi_\tau \colon F^{-1}(U) \to F^{-1}(U)$ is a symplectomorphism if and only if $\tau$ is closed.
\end{lemma}

As a result, $2\pi \Period(U) \subset Z^1(U)$, the space of closed $1$-forms on $U$ for any open $U \subset B$, and we can naturally define $\der$ in $(\Omega^1 / 2\pi \Period)(B)$ with kernel $(Z^1 / 2\pi \Period)(B)$.

\begin{corollary} \label{cor:symplecto-auto-singular}
  Let $(M, \omega, F)$ be a complete integrable system and $B = F(M)$.
  Let $\tau \in (\Omega^1 / 2\pi \Period)(B)$.
  Then $\Psi_\tau \colon M \to M$ is a symplectomorphism if and only if $\tau$ is closed.
\end{corollary}

\begin{lemma} \label{lem:symplecto-endo}
  Let $(M, \omega, F)$ and $(M', \omega', F')$ be complete integrable systems and set $B = F(M)$, $B' = F'(M')$.
  Suppose that $(M, \omega, F)$ admits a Lagrangian section $P \colon B \to M$.
  Let $\varphi \colon M \to M'$ and $G \colon B \to B'$ be diffeomorphisms with $\varphi$ lifting $G$.
  Then $\varphi$ is a symplectomorphism if and only if for any $\tau' \in \Omega^1(B)$, we have $\varphi \circ \Psi_{G^* \tau'} = \Psi_{\tau'} \circ \varphi$ and $\varphi \circ P \circ G^{-1}$ is a Lagrangian section of $F'$. 
\end{lemma}

The next lemma follows from Duistermaat~\cite{MR596430}.

\begin{lemma} \label{lem:symplecto-extension}
  Let $(M, \omega, F)$ and $(M', \omega', F')$ be vertically transitive integrable systems as in \cref{def:vertically-transitive}.
  Let $B = F(M)$, $B' = F'(M')$.
  Suppose $\varphi_P$ is a diffeomorphism from a Lagrangian section $P$ of $F$ to a Lagrangian section $P'$ of $F'$ lifting a diffeomorphism $G \colon B \to B'$.
  If $(G^{-1})^* \Period^{(M, \omega, F)} \subset \Period^{(M', \omega', F')}$, then $\varphi_P$ has a unique extension as a surjective local diffeomorphism $\varphi \colon M \to M'$ such that $(\varphi, G)$ is a morphism of integrable systems.
  If $(G^{-1})^* \Period^{(M, \omega, F)} = \Period^{(M', \omega', F')}$, the pair $(\varphi, G)$ is an isomorphism.
\end{lemma}

\begin{proof}
  Fix an $x \in M$ and let $b = F(x)$ and $b' = G(b)$.
  Let $Q \colon B \to M$ be a smooth section of $F$ through $x$.
  Since $(M, \omega, F)$ is vertically transitive, we have a translation form $\tau^{PQ} \in (\Omega^1 / 2\pi \Period)(B)$.
  Suppose $(G^{-1})^* \Period^{(M, \omega, F)} \subset \Period^{(M', \omega', F')}$.
  Then $(\tau')^{PQ} = (G^{-1})^* \tau^{PQ}$ belongs to $(\Omega^1 / 2\pi \Period)(B')$ and since $(M', \omega', F')$ is also vertically transitive, $\Psi_{(\tau')^{PQ}} \colon M' \to M'$ is a well-defined diffeomorphism.
  Let $\varphi_Q = \Res{\Psi_{(\tau')^{PQ}} \circ \varphi_P \circ P \circ F}_{Q(B)} \colon Q(B) \to M'$.

  Let $R$ be another smooth section of $F$, let $y = R(b)$, and let $\varphi_R$ be defined analoguously to $\varphi_Q$.
  Then $\varphi_Q(x) = \varphi_R(y)$ is equivalent to $(\tau')^{PQ} = (\tau')^{PR} \in (\Omega^1 / 2\pi \Period)(B')$.
  If $x = y$, then $\tau^{PQ} = \tau^{PR} \in (\Omega^1 / 2\pi \Period)(B)$, so $\varphi_Q(x) = \varphi_R(y)$, which we define to be $\varphi(x)$.
  Then $\varphi \colon M \to M'$ lifts $G$ and $(\varphi, G)$ is a morphism of integrable systems.
  By its defining formula, $\varphi$ is a local diffeomorphism.
  The surjectivity of $\varphi$ is due to the vertical transitivity of $(M', \omega', F')$.
  Applying \cref{lem:symplecto-endo} locally, $\varphi^* \omega' = \omega$ (the completeness condition of \cref{lem:symplecto-endo} is not used here).

  When $(G^{-1})^* \Period^{(M, \omega, F)} = \Period^{(M', \omega', F')}$, $(\tau')^{PQ} = (\tau')^{PR} \in (\Omega^1 / 2\pi \Period)(B')$ would imply $\tau^{PQ} = \tau^{PR} \in (\Omega^1 / 2\pi \Period)(B)$, then $x = y$.
  In this case, $\varphi$ is injective.
\end{proof}

\subsection{Local and semiglobal normal forms} \label{ssec:normal-form}

The goal of this subsection is to recall the existence of local action-angle coordinates \emph{\`a la Duistermaat}~\cite{MR596430}, explicit calculations on the Eliasson's local normal form of focus-focus invariants, and classification of automorphisms of the local normal form from the literature.
They are collected here in a systematic and self-contained way.

Throughout this paper, we use $\mathbb{S}^1 = \R / 2\pi\Z$ and $\mathbb{T}^n = (\mathbb{S}^1)^n$ for $n \in \N$.
Let $(M, \omega, F) \in \IntSys$, let $B = F(M)$, and let $B_{\mathrm{r}}$ be the set of regular values of $F$ in $B$.
The next lemma follows from the ideas in Duistermaat~\cite{MR596430}, of which we include a proof for completeness.

\begin{lemma} \label{lem:period-lattice}
  If $U \subset B_{\mathrm{r}}$ is a simply connected open set over which $F$ has compact and connected fibers and admits a Lagrangian section, then there are $\alpha_1, \dotsc, \alpha_n \in Z^1(U)$ such that $\Period(U) = \oplus_{i=1}^n \alpha_i \Z$.
  Moreover, for any $b \in U$, $2\pi \alpha_1(b), \dotsc, 2\pi \alpha_n(b)$ form a $\Z$-basis of the isotropy subgroup of $T^*_bB$ under the action of $\Psi$.
\end{lemma}

\begin{proof}
  Let $b \in B_{\mathrm{r}}$.
  Consider the action of $T^*_bB$ on $F^{-1}(b)$ by $\Psi$.
  Since $F^{-1}(b)$ consists of regular points, the action is locally free, and the orbits are open.
  Since $F^{-1}(b)$ is connected, the action is transitive.
  Since $F^{-1}(b)$ is compact, the isotropy subgroup of this action has to be an $n$-lattice, in which case $F^{-1}(b)$ is diffeomorphic to an $n$-torus.

  Choose a Lagrangian section $P \colon U \to F^{-1}(U)$ and the map
  \begin{align*}
    \lambda \colon T^*U &\to F^{-1}(U) \\
    \beta_b \in T^*_bU &\mapsto \Psi_{2\pi \beta_b} P(b)
  \end{align*}
  is smooth by the smooth dependence of ordinary differential equations on parameters, and we have $F \circ \lambda = \pi$, where $\pi \colon T^*U \to U$ is the projection.
  Then $L_U = \lambda^{-1}(P(U))$ is a closed submanifold of $T^*U$.
  Any $\beta_b \in L_U$ has an open neighborhood on which $\lambda$ is diffeomorphic to its image.
  Let $\alpha_{1, b}, \dotsc, \alpha_{n, b}$ be a $\Z$-basis of $L_U \cap T^*_bU$, the isotropy subgroup divided by $2\pi$, we obtain an open neighborhood $U_b$ of $b$ and $\alpha_1, \dotsc, \alpha_n \in Z^1(U_b)$ such that $\alpha_i(b) = \alpha_{i, b}$ and the images of $\alpha_i$, $i = 1, \dotsc, n$, as sections, coincide with $L_U$ near $\alpha_{i, b}$.
  The closedness of $L_U$ ensures that $L_U \cap T^*_bU_b$ is exactly the union of the images of $\alpha_i$, $i = 1, \dotsc, n$.
  Hence $\Res{\lambda}_{L_U} \colon L_U \to P(U)$ is a smooth covering map, which is a trivial covering by the simple connectedness of $U \subset B_{\mathrm{r}}$.
  We arrive at the conclusions in the statement.
\end{proof}

\begin{theorem}[Action-angle coordinates~\cite{MR997295,mineursystemes}] \label{thm:action-angle}
  Let $U \subset B_{\mathrm{r}}$ be a simply connected open subset over which $F$ has compact and connected fibers and admits a Lagrangian section.
  Let $\Pa{\alpha_1, \dotsc, \alpha_n}$ be a $\Z$-basis of $\Period(U)$.
  There are coordinate systems $\Pa{A_1, \dotsc, A_n} \colon U \to \R^n$ and $\Pa{\theta_1, \dotsc, \theta_n, a_1, \dotsc, a_n} \colon F^{-1}(U) \to \mathbb{T}^n \times \R^n$ such that
  \begin{itemize}
    \item $\der A_i = \alpha_i$;
    \item $a_i = F^* A_i$;
    \item $\omega = \sum_{i=1}^n \der \theta_i \wedge \der a_i$.
  \end{itemize}
  We call $A_i$ the \emph{action integrals}, $a_i$ the \emph{action coordinates} and $\theta_i$ the \emph{angle coordinates}.
\end{theorem}

\begin{definition} \label{def:local-model-focus}
  Let $(x_1, \xi_1, x_2, \xi_2)$ be the coordinates of $\R^4$.
  Let $\omega_0 = \der x_1 \wedge \der \xi_1 + \der x_2 \wedge \der \xi_2$ be the standard symplectic form on $\R^4$.
  Let $q = (q_1, q_2) \colon \R^4 \to \R^2$ be
  \begin{align*}
    q_1 &= x_1 \xi_2 - x_2 \xi_1, &q_2 &= x_1 \xi_1 + x_2 \xi_2.
  \end{align*}
  We call $(\R^4, \omega_0, q)$ the \emph{local normal form of focus-focus singular points}.
\end{definition}

Now we compute the action $\Psi$ associated with $(\R^4, \omega_0, q)$.
Let $z = x_1 + \imag x_2$, $\zeta = \xi_2 + \imag \xi_1$, then $q_1 + \imag q_2 = z\zeta$, and
\begin{align*}
  \mathcal{X}_{q_1} &= -\omega_0^{-1} \der q_1 = x_2 \partial_{x_1} - x_1 \partial_{x_2} + \xi_2 \partial_{\xi_1} - \xi_1 \partial_{\xi_2}, \\
  \mathcal{X}_{q_2} &= -\omega_0^{-1} \der q_2 = - x_1 \partial_{x_1} - x_2 \partial_{x_2} + \xi_1 \partial_{\xi_1} + \xi_2 \partial_{\xi_2}.
\end{align*}
Let $c = (c_1, c_2)$ be the coordinates of $\R^2$, and let $(t_1, t_2) \in \R^2$.
Then the action of $\Omega^1(\R^2)$ is
\begin{equation} \label{eq:flow-euclidean}
  \Psi_{t_1 \der c_1 + t_2 \der c_2} \Pa{z, \zeta} = \Pa{e^{-t_2 - \imag t_1} z, e^{t_2 + \imag t_1} \zeta}.
\end{equation}
Then $(\R^4, \omega_0, q)$ is a complete integrable system whose period sheaf $\Period^{(\R^4, \omega_0, q)}$ is the constant sheaf with stalks $(\der c_1) \Z$.
We will use the identifications
\begin{align*}
  \R^4 &\to \C^2, & \R^2 &\to \C, \\
  (x_1, \xi_1, x_2, \xi_2) &\mapsto (z, \zeta), &(c_1, c_2) &\mapsto c = c_1 + \imag c_2
\end{align*}
throughout this paper.

Let $\R^2_{\mathrm{r}} \simeq \C_{\mathrm{r}} = \Set{c \in \C \mmid c \neq 0}$.
Let $P, Q \colon \R^2 \to \R^4$ be the two Lagrangian sections of $q$ defined by $P(c) = (1, c)$, $Q(c) = (c, 1)$.
Then let $\kappa \in (Z^1 / 2\pi \Period)(\R^2_{\mathrm{r}})$ denote the translation form
\begin{equation} \label{eq:kappa-definition}
  \kappa = \tau^{PQ} = -\Im \ln c \der c_1 - \Re \ln c \der c_2.
\end{equation}
Define subsets of $\R^4 \simeq \C^2$ as follows:
\begin{align*}
  D^0_{\mathrm{u}} &= \Set{(z, \zeta) \in \C^2 \mmid z = 0}, &D^0_{\mathrm{s}} &= \Set{(z, \zeta) \in \C^2 \mmid \zeta = 0}, \\
  \R^4_{\mathrm{r}} = \C^2_{\mathrm{r}} &= \Set{(z, \zeta) \in \C^2 \mmid q(z, \zeta) \neq 0}, &\mathcal{F}_0 &= \Set{(z, \zeta) \in \C^2 \mmid q(z, \zeta) = 0}.
\end{align*}

Here $D^0_{\mathrm{u}}$ and $D^0_{\mathrm{s}}$ are respectively, the unstable and the stable manifolds of $(0, 0)$ under the flow of $\mathcal{X}_{\der c_2}$.
For any $(t_1, t_2) \in \R^2$ with $t_2 > 0$, the origin is the only $\upalpha$-limit point for the flow lines of $\mathcal{X}_{t_1 \der c_1 + t_2 \der c_2}$ in $D^0_{\mathrm{u}}$, and the $\upomega$-limit point for the flow lines in $D^0_{\mathrm{s}}$.
Let $\proj_1, \proj_2 \colon \R^2 \to \R$ be respectively the projection onto the first and the second component.

Let $M$ and $B$ be smooth manifolds and $b \in B$.
Let $\mathcal{N}(B, b)$ denote the collection of neighborhoods of $b$ in $B$.
If $F \colon M \to B$ is a surjection and $\mathcal{F}$ is a fiber of $F$, then let $\mathcal{N}_F(M, \mathcal{F})$ denote the collection of saturated neighborhoods of $\mathcal{F}$ in $M$ with respect to $F$, where a subset $W \subset M$ is \emph{saturated} with respect to $F$ if $F^{-1}(F(W)) = W$.

\begin{definition} \label{def:focus-focus}
  A singular point $m \in M$ of $F$ is of \emph{focus-focus type} if the Hessians of the components of $F$ under some sympletic coordinates of $M$ near $m$ and some smooth coordinates of $F(M)$ near $F(m)$ equals those of $q$ near $0 \in \R^4$ which is given in \cref{def:local-model-focus}.
\end{definition}

\begin{definition} \label{def:eliasson-focus}
  An \emph{Eliasson local chart} at a singular point $m \in M$ of $F$ of focus-focus type is an isomorphism $(\psi, E)$ from the integrable system $(V, \omega, \Res{F}_V)$ to $(V_0, \omega_0, \Res{q}_{V_0})$ where $V \in \mathcal{N}_F(M, m)$, $V_0 \in \mathcal{N}_q(\R^4, 0)$, $U \in \mathcal{N}(B, 0)$, $U_0 \in \mathcal{N}(\R^2, 0)$, and $q$ is given in \cref{def:local-model-focus}.
  In other words, the following diagram commutes:
  \begin{equation*}
    \xymatrix{
      (V, \omega) \ar[r]^{\psi} \ar[d]^{F} & (V_0, \omega_0) \ar[d]^{q} \\
      U \ar[r]^{E} & U_0
    }.
  \end{equation*}
\end{definition}

\begin{theorem}[Eliasson's theorem~\cite{Eliasson_1984,MR3098197,MR3098203}] \label{thm:eliasson-focus}
  There is an Eliasson local chart at any singular point of focus-focus type.
\end{theorem}

Let $\varphi_X, \varphi_Y \colon (\R^4, \omega_0) \to (\R^4, \omega_0)$ and $G_X, G_Y \colon \R^2 \to \R^2$ be
\begin{equation} \begin{aligned} \label{eq:local-model-diffeomorphism}
   \varphi_X(z, \zeta) &= (\imag \cj{z}, \imag \cj{\zeta}), & G_X(c) &= -\cj{c}, \\
   \varphi_Y(z, \zeta) &= (\imag \cj{\zeta}, -\imag \cj{z}), & G_Y(c) &= \cj{c}.
\end{aligned} \end{equation}
Then $\varphi_X, \varphi_Y$ are symplectomorphisms lifting $G_X$, $G_Y$ respectively (with respect to $q$) and hence they form automorphisms of the standard local model $(\R^4, \omega_0, q)$.

Define a function
\begin{equation} \begin{aligned} \label{eq:def-r}
   r \colon \R^4 \setminus D^0_{\mathrm{u}} \simeq \C^2 \setminus D^0_{\mathrm{u}} &\to \R, \\
   r(z, \zeta) &= \ln \abs{z}
\end{aligned} \end{equation}
as a measurement of the fiberwise translation.
In fact, if $\tau = t_1 \der c_1 + t_2 \der c_2$, we have
\begin{equation} \begin{aligned} \label{eq:r-translation}
  r \circ \Psi_\tau &= r - t_2, \\
  r \circ \Psi_\kappa &= r + \ln \abs{q}.
\end{aligned} \end{equation}

Note that the integrable system $(\R^4, \omega_0, q)$ can be restricted onto $\R^4 \setminus D^0_{\mathrm{u}}$ and $\R^4 \setminus D^0_{\mathrm{s}}$ respectively and remains complete.

\begin{lemma} \label{lem:Psi-kappa}
  The map $\Psi_\kappa \colon (\R^4_{\mathrm{r}}, \omega_0) \to (\R^4_{\mathrm{r}}, \omega_0)$ as in \cref{cor:symplecto-auto-singular}, where $\kappa$ is defined in \cref{eq:kappa-definition}, can be extended to a symplectomorphism $\wt{\Psi}_\kappa \colon (\R^4 \setminus D^0_{\mathrm{u}}, \omega_0) \to (\R^4 \setminus D^0_{\mathrm{s}}, \omega_0)$.
\end{lemma}

\begin{proof}
  Since the map
  \begin{equation} \begin{split} \label{eq:Psi-kappa-ext}
    \wt{\Psi}_\kappa \colon \R^4 \setminus D^0_{\mathrm{u}} &\to \R^4 \setminus D^0_{\mathrm{s}}, \\
    (z, \zeta) &\mapsto (z^2 \zeta, z^{-1}),
  \end{split} \end{equation}
  coincides with $\Psi_\kappa$ on $\R^4_{\mathrm{r}}$, $\wt{\Psi}_\kappa$ is an extension of $\Psi_\kappa$ as a diffeomorphism.
  Since $\der \kappa = 0$ in $\R^2_{\mathrm{r}}$, by \cref{cor:symplecto-auto-singular}, $\Psi_\kappa$ is symplectomorphism of $(\R^4_{\mathrm{r}}, \omega_0)$.
  By continuity, $\wt{\Psi}_\kappa$ is a symplectomorphism.
  Alternatively, one can verify $\wt{\Psi}_\kappa^* \omega_0 = \omega_0$ by explicit computations.
\end{proof}

\begin{lemma} \label{lem:kappa-lemma}
  Let $G \colon U \to U'$ be a diffeomorphism where $U, U' \in \mathcal{N}(\R^2, 0)$ such that $G(c_1, c_2) = (c_1, c_2 + \inftes)$.
  Then $G^*\kappa = \kappa + \inftes \der c_1 + \inftes \der c_2$ as an element of $(\Omega^1 / 2\pi \Period)(U \cap \R^2_\mathrm{r})$.
\end{lemma}

\begin{proof}
  For $c \neq 0$,
  \begin{equation} \begin{split} \label{eq:kappa-lemma}
    G^* \kappa(c) - \kappa(c) &= -\ln \abs{G(c)} \, \frac{\partial G_2}{\partial c_1}(c) \der c_1 - \arg \frac{G(c)}{c} \der c_1 \\
    &\phantom{{}=} - \ln \abs{\frac{G(c)}{c}} \, \frac{\partial G_2}{\partial c_2}(c) \der c_2 - \ln \abs c \Pa{\frac{\partial G_2}{\partial c_2}(c) - 1} \der c_2.
  \end{split} \end{equation}

  By the fact of $x \mapsto \ln(1 + x)$ being analytic and $c \mapsto \frac{G(c)}{c} - 1$ being flat, both components of $\ln \frac{G}{c}$ are flat, which is to say that $-\ln \abs{\frac{G}{c}}, \arg \frac{G}{c} \in \inftes$.
  Since $\frac{\partial G_2}{\partial c_2} - 1, \frac{\partial G_2}{\partial c_1} \in \inftes$, by \cref{lem:logarithm-flat}, we have $\ln \abs \cdot \Pa{\frac{\partial G_2}{\partial c_2} - 1}, \ln \abs{G} \, \frac{\partial G_2}{\partial c_1} \in \inftes$.
  Hence the form in \cref{eq:kappa-lemma} has the shape of $\inftes \der c_1 + \inftes \der c_2$.
\end{proof}

\begin{lemma} \label{lem:Psi-G-pullback-kappa}
  Let $G \colon U \to U'$ be a diffeomorphism where $U, U' \in \mathcal{N}(\R^2, 0)$ such that $G(c_1, c_2) = (c_1, g(c_1, c_2))$ with $\frac{\partial g}{\partial c_2} > 0$.
  Note that $G^* (\Res{\kappa}_{U' \cap B_\mathrm{r}}) \in (\Omega^1 / 2\pi \Period)(U \cap B_\mathrm{r})$.
  Then the symplectomorphism $\Psi_{G^* \kappa} \colon (q^{-1}(U \cap B_\mathrm{r}), \omega_0) \to (q^{-1}(U' \cap B_\mathrm{r}), \omega_0)$ can be extended to a symplectomorphism 
  \begin{equation*}
    \wt{\Psi}_{G^* \kappa} \colon (q^{-1}(U) \setminus D^0_{\mathrm{u}}, \omega_0) \to (q^{-1}(U') \setminus D^0_{\mathrm{s}}, \omega_0)
  \end{equation*} 
  if and only if $G(c_1, c_2) = (c_1, c_2 + \inftes)$.
\end{lemma}

\begin{proof}
  Here we notice that $q^{-1}(U \cap B_\mathrm{r}) = q^{-1}(U) \cap \R^4_{\mathrm{r}}$ is a punctured neighborhood of $\mathcal{F}_0$ in $\R^4$ and
  \begin{align*}
    q^{-1}(U) \setminus D^0_{\mathrm{u}} \in \mathcal{N}_q(\R^4 \setminus D^0_{\mathrm{u}}, D^0_{\mathrm{s}}), \qquad q^{-1}(U') \setminus D^0_{\mathrm{s}} \in \mathcal{N}_q(\R^4 \setminus D^0_{\mathrm{s}}, D^0_{\mathrm{u}}).
  \end{align*}

  Recall that
  \begin{equation*}
    G^* \kappa(c) = -\ln \abs{G(c)} \frac{\partial G_2}{\partial c_1}(c) \der c_1 - \arg G(c) \der c_1 - \ln \abs{G(c)} \frac{\partial G_2}{\partial c_2}(c) \der c_2.
  \end{equation*}
  Let $(z, \zeta) = P(c) = (1, c)$, so $c = q(z, \zeta)$.
  Let $U$ be a neighborhood of $0$ in $\C$.
  Let $h_1 \colon U \cap \C_{\mathrm{r}} \to \C$ and $h_2 \colon U \to \C$ be
  \begin{align*}
    h_1(c) &= \frac{G(c)}{c}, &h_2(c) &= \frac{\partial G_2}{\partial c_2}(c) - 1 + \imag \frac{\partial G_2}{\partial c_1}(c).
  \end{align*}
  Then for $c \in U \cap \C_{\mathrm{r}}$ we have
  \begin{align*}
    \Psi_{G^* \kappa} \circ P(c) = \Pa{\abs{G(c)}^{h_2(c)} G(c), \abs{G(c)}^{-h_2(c)} h_1(c)^{-1}}.
  \end{align*}

  Suppose $\Psi_{G^* \kappa}$ can be extended to $\wt{\Psi}_{G^* \kappa} \colon (\R^4 \setminus D^0_{\mathrm{u}}, \omega_0, D^0_{\mathrm{s}}) \to (\R^4 \setminus D^0_{\mathrm{s}}, \omega_0, D^0_{\mathrm{u}})$.
  By continuity, we conclude that $\lim_{c \to 0} \abs{G(c)}^{-h_2(c)} h_1(c)^{-1}$ as the second complex component of $\wt{\Psi}_{G^* \kappa}(1, 0)$ is nonzero.
  For any fixed $c \in \C_{\mathrm{r}}$, $t \mapsto h_1(tc)$ is smooth at $0$ and $\lim_{t \to 0} h_1(tc) \neq 0$.
  The map $t \mapsto \abs{t}^{-h_2 \circ G^{-1}(tc)}$ is smooth and has nonzero limit at $0$.
  So $t \mapsto h_2 \circ G^{-1}(tc) \ln \abs{tc} = h_2 \circ G^{-1}(tc) \ln \abs{t} + \mathrm{C}^\infty$ is smooth at $0$.
  Hence by an analogous $1$-dimensional version of \cref{lem:logarithm-flat}, $t \mapsto h_2 \circ G^{-1}(tc)$ is flat at $0$.
  By arbitrarity of $c$, we have $h_2 \circ G^{-1} \in \inftes$, so $h_2 \in \inftes$.
  Therefore $G(c_1, c_2) = (c_1, c_2 + \inftes)$.

  On the other hand, if it is known that $G(c_1, c_2) = (c_1, c_2 + \inftes)$, then $h_1$ can be extended to $0$ such that $h_1(0) \neq 0$ and $h_2 \in \inftes$.
  Moreover, $h_2 \ln \abs{G} \in \inftes$ by \cref{lem:logarithm-flat}, so $\abs{G}^{h_2}$ can be extended to a smooth function with value $1$ at $0$.
  Then $\Psi_{G^* \kappa}$ can be extended to a diffeomorphism $\wt{\Psi}_{G^* \kappa} \colon \R^4 \setminus D^0_{\mathrm{u}} \to \R^4 \setminus D^0_{\mathrm{s}}$ at $D^0_{\mathrm{s}}$ sending $D^0_{\mathrm{s}}$ to $D^0_{\mathrm{u}}$.
  The pair $(\wt{\Psi}_{G^* \kappa}, G)$ is an isomorphism since $\wt{\Psi}_{G^* \kappa}$ is a symplectomorphism on the part of its domain inside of $\R^4_{\mathrm{r}}$.
\end{proof}

\begin{lemma} \label{lem:varphi_G}
  Let $G \colon U \to U'$ be a diffeomorphism where $U, U' \in \mathcal{N}(\R^2, 0)$ such that $G(c_1, c_2) = (c_1, g(c_1, c_2))$ with $\frac{\partial g}{\partial c_2} > 0$.
  Then there is a unique symplectomorphism
  \begin{align*}
    \varphi_G \colon (q^{-1}(U) \setminus D^0_{\mathrm{u}}, \omega_0) &\to (q^{-1}(U') \setminus D^0_{\mathrm{u}}, \omega_0)
  \end{align*}
  lifting $G$ characterized by $\varphi_G(1, c) = (1, G(c))$ for $c \in U$.
  Or equivalently, $\varphi_G$ is characterized by 
  \begin{equation} \label{eq:r-varphi_G}
    r = \Pa{\frac{\partial g}{\partial c_2} \circ q} \cdot \Pa{r \circ \varphi_G}
  \end{equation}
  in $q^{-1}(U)$.
  If $G(c_1, c_2) = (c_1, c_2 + \inftes)$, then $\varphi_G$ can be uniquely extended to a symplectomorphism
  \begin{align*}
    \wt{\varphi}_G \colon (q^{-1}(U), \omega_0) &\to (q^{-1}(U'), \omega_0).
  \end{align*}
\end{lemma}

\begin{proof}
  The first part is a result of \cref{lem:symplecto-extension}, as $c \mapsto (1, c)$ and $c \mapsto (1, G(c))$ are Lagrangian sections of $q$, and $G^* \der c_1 = \der c_1$.
  If $\varphi_G \colon (\R^4 \setminus D^0_{\mathrm{u}}, \omega_0, D^0_{\mathrm{s}}) \to (\R^4 \setminus D^0_{\mathrm{u}}, \omega_0, D^0_{\mathrm{s}})$ is a symplectomorphism lifting $G$ and \cref{eq:r-varphi_G}, then $\varphi_G(1, c) = (1, G(c))$ holds automatically.
  We then show that the $\varphi_G$ given by $\varphi_G(1, c) = (1, G(c))$ has the property \cref{eq:r-varphi_G}.
  In fact, if $(z, \zeta) \in \R^4_{\mathrm{r}}$ in a neighborhood of $0$ in $\R^4$, $c = q(z, \zeta) \in \R^2_{\mathrm{r}}$ and let $\tau' = t'_1 \der c_1 + t'_2 \der c_2 \in \Omega^1(\R^2)$ be so that $\Psi_{\tau'}(1, G(c)) = (z', \zeta')$, then since $\varphi_G \circ \Psi_{G^* \tau'} = \Psi_{\tau'} \circ \varphi_G$, we have by \cref{lem:symplecto-auto}
  \begin{align*}
    r(z, \zeta) = r \circ \Psi_{G^* \tau'}(1, c) = -\frac{\partial g}{\partial c_2}(c) t'_2(c) = \frac{\partial g}{\partial c_2}(c) \cdot \Pa{r \circ \Psi_{\tau'}(1, G(c))} = \frac{\partial g}{\partial c_2}(c) \cdot \Pa{r \circ \varphi_G(z, \zeta)}.
  \end{align*}

  For the second part, consider
  \begin{equation} \begin{split} \label{eq:varphi-G-1}
    \wt{\varphi}_G \colon q^{-1}(U) &\to q^{-1}(U'), \\
    (z, \zeta) &\mapsto \Pa{z e^{\Pa{\frac{\partial g}{\partial c_2} \circ q(z, \zeta) - 1 + \imag \frac{\partial g}{\partial c_1} \circ q(z, \zeta)} \ln \abs z}, \textstyle\frac{G \circ q(z, \zeta)}{q(z, \zeta)} \zeta e^{-\Pa{\frac{\partial g}{\partial c_2} \circ q(z, \zeta) - 1 + \imag \frac{\partial g}{\partial c_1} \circ q(z, \zeta)} \ln \abs z}}.
  \end{split} \end{equation}
  Since $\frac{\partial g}{\partial c_2} \circ q - 1 + \imag \frac{\partial g}{\partial c_1} \circ q$ is flat at the origin, the exponents in \cref{eq:varphi-G-1} are flat at the origin by \cref{lem:logarithm-flat} and therefore \cref{eq:varphi-G-1} is smooth.
  By explicit calculations using \cref{lem:symplecto-endo}, $\wt{\varphi}_G$ extends $\varphi_G$.
  By continuity, the extension is a symplectomorphism and such an extension is unique.
\end{proof}

\begin{lemma} [{\cite[Lemma 4.1, Lemma 5.1]{MR1941440}}] \label{lem:local-flat}
  Let $G \colon U \to U'$ be a diffeomorphism where $U, U' \in \mathcal{N}(\R^2, 0)$.
  Then there is a symplectomorphism $\varphi$ in a neighborhood of the origin in $\R^4$ lifting $G$ if and only if $G(c_1, c_2) = (e_1 c_1, e_2 c_2 + \inftes)$, with $e_i = \pm 1$, $i = 1, 2$.
\end{lemma}

\begin{proof}
  If such a $\varphi \colon (V, \omega_0) \to (V', \omega_0)$ exists, by possibly shrinking $U$ and $V$, both $\Period^{(V, \omega_0, \Res{q}_{V})}(U)$ and $\Period^{(V', \omega_0, \Res{q}_{V'})}(U')$ would have rank one, generated by $\der c_1$ on $U$ and $U'$ respectively.
  But since $G$ is a diffeomorphism and $G^* \Period^{(V', \omega_0, \Res{q}_{V'})}(U') = \Period^{(V, \omega_0, \Res{q}_{V})}(U)$, we must have $G(c_1, c_2) = (e_1 c_1, g(c_1, c_2))$, where $e_1 = \pm 1$, for some smooth function $g \colon U \to \R$.
  Let $W = q^{-1}(U)$ and $W' = q^{-1}(U')$.
  Now we \emph{de facto} have $G^* \Period^{(W', \omega_0, \Res{q}_{W'})} = \Period^{(W, \omega_0, \Res{q}_W)}$ and by applying \cref{lem:symplecto-extension} to the vertically transitive integrable systems $(W \setminus D^0_{\mathrm{u}}, \omega_0, \Res{q}_{W \setminus D^0_{\mathrm{u}}})$ and $(W \setminus D^0_{\mathrm{s}}, \omega_0, \Res{q}_{W \setminus D^0_{\mathrm{s}}})$ respectively we can extend $\varphi$ to a symplectomorphism $\wt{\varphi} \colon (W, \omega_0) \to (W', \omega_0)$ lifting $G$.
  Since $\varphi$ preserves the singular fiber $\mathcal{F}_0$, and the punctured fiber $\mathcal{F}_0 \setminus \Set{0}$ has two components $D^0_{\mathrm{u}} \setminus \Set{0}$ and $D^0_{\mathrm{s}} \setminus \Set{0}$, either $\varphi$ preserves the two components or it exchanges them.
  In the first case, $\wt{\varphi}(D^0_{\mathrm{s}}) = D^0_{\mathrm{s}}$; in the latter case, $\wt{\varphi}(D^0_{\mathrm{s}}) = D^0_{\mathrm{u}}$.
  In any of the four cases above ($e_1 = \pm 1$, $\wt{\varphi}(D^0_{\mathrm{s}}) = D^0_{\mathrm{s}}$ or $D^0_{\mathrm{u}})$), there is exactly one choice of $(\varphi_0, G_0)$ from the set (with maps defined in \cref{eq:local-model-diffeomorphism})
  \begin{equation} \label{eq:local-model-automorphism}
    \Set{(\identity, \identity), (\varphi_X, G_X), (\varphi_Y, G_Y), (\varphi_Y \circ \varphi_X, G_Y \circ G_X)}
  \end{equation}
  such that $\varphi_0 \circ \wt{\varphi} \colon (W \setminus D^0_{\mathrm{u}}, \omega_0) \to (W \setminus D^0_{\mathrm{u}}, \omega_0)$, and $\frac{\partial (\proj_2 \circ G_0 \circ G)}{\partial c_2} > 0$.
  Now we have that
  \begin{align*}
    (\varphi_0 \circ \wt{\varphi})^{-1} \circ \Psi_\kappa \circ (\varphi_0 \circ \wt{\varphi}) = \Psi_{(G_0 \circ G)^* \kappa} \colon (W \setminus D^0_{\mathrm{u}}, \omega_0) \to (\R^4 \setminus D^0_{\mathrm{s}}, \omega_0)
  \end{align*}
  is a symplectomorphism and $q \circ (\varphi_0 \circ \wt{\varphi}) = (G_0 \circ G) \circ q$.
  By \cref{lem:Psi-G-pullback-kappa}, we have $G_0 \circ G(c_1, c_2) = (c_1, c_2 + \inftes)$.
  Therefore, $G(c_1, c_2) = (e_1 c_1, e_2 c_2 + \inftes)$, with $e_i = \pm 1$, $i = 1, 2$.

  Conversely, if $G(c_1, c_2) = (e_1 c_1, e_2 c_2 + \inftes)$ we assume, without loss of generality, that $e_1 = e_2 = 1$.
  Otherwise, we can apply a pair of maps in \cref{eq:local-model-automorphism}.
  Let $\varphi = \wt{\varphi}_G \colon q^{-1}(U) \to q^{-1}(U')$ be the symplectomorphism defined as in \cref{lem:varphi_G}.
  Then we have $q \circ \varphi = G \circ q$.
\end{proof}

For the sake of making this paper self-contained, we now recall the known results on the topological structure near the focus-focus fiber up to fiber-preserving diffeomorphisms; see \cite[Chapter 9.8 and Lemma 9.8]{MR2036760} and \cite[Section 3]{MR1454718} for reference.

\begin{definition} \label{def:int-sys-ff}
  Let $\IntSysFF$ be the collection of $4$-dimensional integrable systems $(M, \omega, F) \in \IntSys$ such that $F$ is proper onto its image and has connected fibers\footnotemark{} one of which is a singular fiber $\mathcal{F} = F^{-1}(0)$ that contains focus-focus singular points of $F$.
  We further assume that there are no other singular points in a saturated neighborhood of $\mathcal{F}$.
\end{definition}

\footnotetext{In this case, by the local models, $F$ is an open map and the fibers of $F$ being connected implies that the preimage of any connected set under $F$ is connected.}

The assumptions in \cref{def:int-sys-ff} are not too restrictive since, by the local normal form, focus-focus singular points are isolated.
In fact, by \cite[Proposition 3.21]{MR3843843} an almost-toric integrable system $(M, \omega, F)$ (in the sense that $F$ is proper onto its image and all of its singular orbits are compact and non-degenerate without hyperbolic blocks) satisfies the assumptions of $\IntSysFF$ if the fibers of $F$ are connected.
A semitoric integrable system also belongs to $\IntSysFF$, as defined in \cite[Definition 2.1]{MR2534101} in the sense that the first component of the momentum map is proper and generates a Hamiltonian circle action on $M$, and the momentum map has only non-degenerate singular points without hyperbolic blocks.

Let $(M, \omega, F) \in \IntSysFF$.
Let $\Crit(\mathcal{F})$ denote the finite set of singular points of $F$ in $\mathcal{F} = F^{-1}(0)$ whose cardinality is called the \emph{multiplicity} of $\mathcal{F}$.
For $k \in \N$, let $\IntSysFFk{k}$ be the collection of $(M, \omega, F) \in \IntSysFF$ where $\mathcal{F}$ has multiplicity $k$.
For $(M, \omega, F) \in \IntSysFFk{k}$, as shown in Zung~\cite[Theorem 5.1]{MR1389366}, $\mathcal{F}$ is homeomorphic to a $2$-torus pinched $k$ times along $k$ homologous $1$-cycles with an infinite cyclic group as the isotropy group of the $T^*_0 B$-action on $\mathcal{F} \setminus \Crit(\mathcal{F})$ by $\Psi$, where $B = F(M)$.

Throughout this paper, we denote by $\Z_k$, $k \in \N$, the quotient group $\Z / k\Z$ of residue classes modulo $k$ with the induced operation from the addition on $\Z$.
Implied by \cite[Theorem 5.1]{MR1389366}, there is a line in $T^*_0B$ acting on $\mathcal{F} \setminus \Crit(\mathcal{F})$ as $\mathbb{S}^1$.
We show that such an $\mathbb{S}^1$-action can be extended onto a saturated neighborhood of $\mathcal{F}$, which follows from \cite[Proposition 3 and Corollary 1]{MR1454718}.

\begin{lemma} \label{lem:semiglobal-semitoricity}
  Let $(M, \omega, F) \in \IntSysFFk{k}$ and $B = F(M)$.
  Then the section space $\Period(B)$ of the sheaf $\Period$ is an infinite cyclic group in $Z^1(B)$, so it can be viewed as a constant sheaf associated to $\Z$ over $B$.
  The quotient sheaf restricted to $B_\mathrm{r}$, $\Res{(\Period / \Period(B))}_{B_\mathrm{r}}$, is also a constant sheaf associated to $\Z$ over $B$.
  In fact, there is an assignment to any simply connected open set $U \subset B_\mathrm{r}$ a generator $\alpha_U$ of the infinite cyclic group $\Period(U) / \Res{\Period(B)}_U$, such that, for any such open sets $U_1$ and $U_2$, the restrictions of $\alpha_{U_1}$ and $\alpha_{U_2}$ to $U_1 \cap U_2$ coincide.
\end{lemma}

There are some further results for the smooth structure of the neighborhoods of singular fibers in~\cite{MR4057723}.

\section{Symplectic classification theorem} \label{sec:invariant-classification}

In this section we state the main theorem of the paper (\cref{thm:invariant-bijection}) formulated by the results from \cref{sec:preliminary}.
As the multiple singular points on a fiber lie in no particular order, one must make a number of choices in the construction of the invariant.
For a precise formulation of our result we need to first deal with issues of orientations and directions.
The result announced in the introduction of the paper is stated precisely as a consequence of our main theorem (\cref{cor:invariant-bijection-quotient}).

\subsection{Directions and singularity atlas} \label{ssec:direction}

Fix a $k \in \N$ and recall \cref{def:int-sys-ff}.
Let $(M, \omega, F) \in \IntSysFFk{k}$, $B = F(M)$, and $\mathcal{F}$ be the singular fiber.
Shrink $M$ to a saturated neighborhood of $\mathcal{F}$ if necessary.
Then by \cref{lem:semiglobal-semitoricity} $\Period(B)$ and $\Res{(\Period / \Period(B))}_{B_\mathrm{r}}$ can be viewed as infinite cyclic groups.

\begin{definition} \label{def:direction}
  A pair $(\alpha_1, \alpha_2)$ is a \emph{direction} of $(M, \omega, F)$ if $\alpha_1$ is a generator of $\Period(B)$ and $\alpha_2$ is an generator of $\Res{(\Period / \Period(B))}_{B_\mathrm{r}}$.
  We call $\alpha_1$ the \emph{$J$-direction} and $\alpha_2$ the \emph{$H$-direction}.
  We denote by $\Dir(M, \omega, F)$ the set of directions of $(M, \omega, F)$.
\end{definition}

Given a direction as in \cref{def:direction} we have
\begin{align*}
  \Period(B) = \Z \alpha_1, \qquad \Res{(\Period / \Period(B))}_{B_\mathrm{r}} = \Z \alpha_2, \qquad \Period(U) = \Z \Res{\alpha_1}_U \oplus \Z \Res{\alpha_2}_U.
\end{align*}
for any simply connected open set $U \subset B_\mathrm{r}$ and note that $\alpha_1$ is a $1$-form on $B$ while $\alpha_2$ represents a $1$-form on open subsets of $B_\mathrm{r}$ modulo integer multiples of $\alpha_1$.

\begin{remark} \label{rem:J-H-semitoric-system}
  The use of the letters $J$ and $H$ is inspired by the notations in semitoric systems where the momentum maps are usually written as $(J, H)$ such that the flow of $\mathcal{X}_J$ is $2\pi$-periodic.
\end{remark}

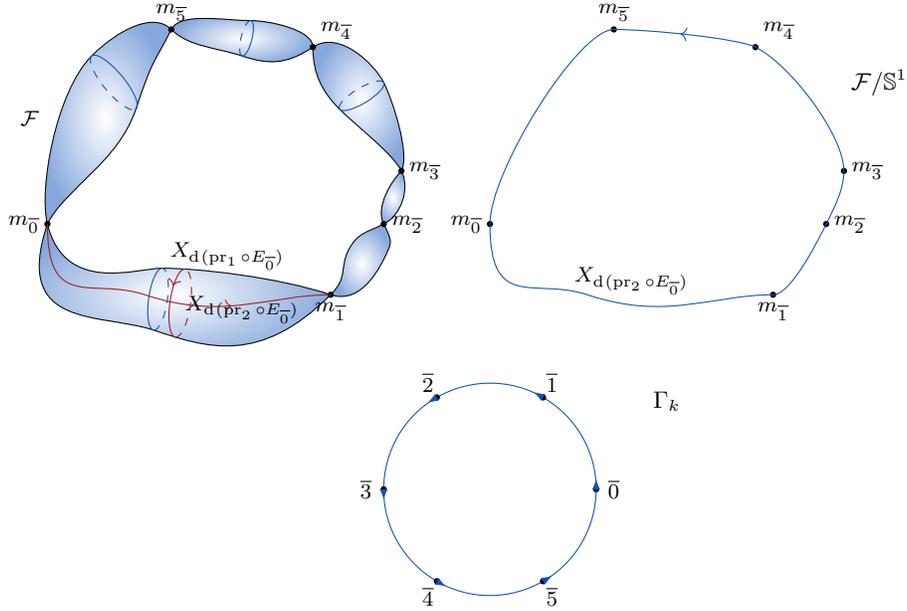
\inputfigure{singular-fiber}{singular-fiber}{Reduction of the singular fiber.
The quotient space $\mathcal{F} / \mathbb{S}^1$, is a circle with $k$ marked points.
Compare $\mathcal{F} / \mathbb{S}^1$ with $\Gamma_k$, which is the cycle graph with vertex set $\Z_k$ and set of edges $\{(j, j+\Zkone) \mid j \in \Z_k\}$.
The automorphism group $D_k$ of $\Gamma_k$ is isomorphic to the group generated by $\gamma_Y$ and $\theta_p$ (see \cref{eq:def-gamma-X,eq:def-gamma-Y,eq:def-theta-p}).
That is not a coincidence.}

The set $\Dir(M, \omega, F)$ contains $4$ different directions.
Recall from \cref{ssec:normal-form} that near any focus-focus singular point $m_j \in M$ there is an Eliasson local chart.

\begin{definition} \label{def:compatible-local-chart}
  Let $(\alpha_1, \alpha_2) \in \Dir(M, \omega, F)$ and let $\Pa{m_j}_{j \in \Z_k} = \Crit(\mathcal{F})$.
  An Eliasson local chart $(\psi_j, E_j)$ near $m_j \in M$ is \emph{compatible with the direction $(\alpha_1, \alpha_2)$} if $E_j^* \der c_1 = \alpha_1$ and $E_j^* \der c_2$ at non-origin points are linear combinations of $\alpha_1$ and (on an open subset any representative of) $\alpha_2$ with positive $\alpha_2$-coefficients.
  A collection $\Pa{(\psi_j, E_j)}_{j \in \Z_k}$ where $(\psi_j, E_j)$ is an Eliasson local chart near $m_j$, $j \in \Z_k$, is called a \emph{singularity atlas} of $(M, \omega, F)$.

  A singularity atlas is \emph{compatible with the direction $(\alpha_1, \alpha_2)$} if for every $j \in \Z_k$, $(\psi_j, E_j)$ is compatible with $(\alpha_1, \alpha_2)$, and for any flow line of $\mathcal{X}_{E_{\Zkzero}^* \der c_1}$ in $\mathcal{F}$, whenever the $\upalpha$-limit point is labeled $m_j$, $j \in \Z_k$, its $\upomega$-limit point is labeled $m_{j+\Zkone}$.
\end{definition}

\begin{lemma} \label{lem:existence-compatible-singularity-atlas}
  Given $(\alpha_1, \alpha_2) \in \Dir(M, \omega, F)$ and $m_{\Zkzero} \in \Crit(\mathcal{F})$, there is a unique way to label $\Crit(\mathcal{F}) = \Pa{m_j}_{j \in \Z_k}$ such that $(M, \omega, F)$ has a singularity atlas $\Pa{(\psi_j, E_j)}_{j \in \Z_k}$ compatible with $(\alpha_1, \alpha_2)$.
\end{lemma}

\begin{proof}
  By possibly composing with one of the pairs of maps in \cref{eq:local-model-automorphism}, we obtain a chart $(\psi_j, E_j)$ near any $m_j$ compatible with $(\alpha_1, \alpha_2)$.
  By \cite[Theorem 5.1]{MR1389366}, the trajectories of $\mathcal{X}_{E_{\Zkzero}^* \der c_2}$ in $\mathcal{F}$ away from $\Crit(\mathcal{F})$ have limits at different critical points (except when $k = 1$).
  On two spheres in $\mathcal{F}$ intersecting at a critical point, the trajectories go to opposite directions relative to that point.
  Hence we can sort $\Pa{m_j}_{j \in \Z_k}$ one by one such that $\Pa{(\psi_j, E_j)}_{j \in \Z_k}$ is compatible with $(\alpha_1, \alpha_2)$.
\end{proof}

\subsection{Construction of the invariants} \label{ssec:invariant-construction}

Let $(M, \omega, F, (\alpha_1, \alpha_2), m_{\Zkzero})$ be a \emph{marked directed integrable system}, which means that $(M, \omega, F) \in \IntSysFFk{k}$ with \emph{direction} $(\alpha_1, \alpha_2) \in \Dir(M, \omega, F)$ and \emph{mark} $m_{\Zkzero} \in \Crit(\mathcal{F})$.
Let $\mathcal{F}$ be the singular fiber.
Let $\Pa{m_j}_{j \in \Z_k} = \Crit(\mathcal{F})$ and let $\Pa{(\psi_j \colon V_j \to \psi_j(V_j), E_j \colon U \to U_j)}_{j \in \Z_k}$ be a singularity atlas compatible with the direction $(\alpha_1, \alpha_2)$.
Let $W = F^{-1}(U)$.

To zoom out from local to semiglobal, we will extend the isomorphism $(\psi_j^{-1}, E_j^{-1})$ along the trajectories.
For $j \in \Z_k$, let $W_j$ be the minimal invariant subset of $M$ by $\Psi$ containing $V_j \in \mathcal{N}(M, m_j)$; in other words,
\begin{equation} \label{eq:def-Mj}
  W_j = \bigcup_{\alpha \in \Omega^1(F(V_j))} \Psi_\alpha(V_j);
\end{equation}
see \cref{fig:injectivity}.
Let $W_{j, j+\Zkone} = W_j \cap W_{j+\Zkone}$; then $W_{j, j+\Zkone}$ is a neighborhood of the orbit of $\Psi$ in $\mathcal{F}$ where $\mathcal{X}_{E_j^* \der c_2}$ flows from $m_j$ to $m_{j+\Zkone}$.
It is important to note that both $(W_j, \omega_0, \Res{F}_{W_j})$ and $(W_{j, j+\Zkone}, \omega_0, \Res{F}_{W_{j, j+\Zkone}})$ are complete but only the latter one is vertically transitive.

\inputfigure{injectivity}{injectivity}{The set $W_{\Zkzero}$ consists of the shaded region in $\mathcal{F}$ and whole regular fibers nearby.}

We construct the first set of the invariants: we split the period form $\alpha_2$ into the singular part ``across singular points'' and the regular part.
The regular part is an invariant.
Recall that local sections of $\Period$ consist of closed $1$-forms and then it makes sense to say whether a local section of $\Omega^1 / 2\pi \Period$ is closed, the subpresheaf of which is denoted by $Z^1 / 2\pi \Period$.

\begin{lemma} \label{lem:definition-sigma}
  The closed section
  \begin{equation} \label{eq:definition-sigma}
    \sigma = -\sum_{j \in \Z_k} E_j^* \kappa = 2\pi \alpha_2 - \sum_{j \in \Z_k} E_j^* \kappa \in (Z^1 / 2\pi \Period)(U \cap B_\mathrm{r})
  \end{equation}
  has a representative in $Z^1(U \cap B_\mathrm{r})$ with a smooth extension in $Z^1(U)$, where $\kappa$ is as defined in \cref{eq:kappa-definition}.
  In other words, there is a $\tilde \sigma \in \Omega^1(U)$ such that for any $b \in U \cap B_\mathrm{r}$ there is a $U_b \in \mathcal{N}(B_\mathrm{r}, b)$ such that $\Res{\tilde \sigma}_{U_b} \in \Res{\sigma}_{U_b} + 2\pi \Period(U_b)$.
\end{lemma}

\begin{proof}
  For $j \in \Z_k$, we can extend $(\psi_j^{-1}, E_j^{-1})$ by \cref{lem:symplecto-extension} to morphisms between vertically transitive integrable systems
  \begin{align*}
    (\lambda_j^+, E_j^{-1}) &\colon (W \setminus D^0_{\mathrm{s}}, \omega_0, \Res{q}_{W \setminus D^0_{\mathrm{s}}}) \to (W_{j, j+\Zkone}, \omega, \Res{F}_{W_{j, j+\Zkone}}), \\
    (\lambda_j^-, E_j^{-1}) &\colon (W \setminus D^0_{\mathrm{u}}, \omega_0, \Res{q}_{W \setminus D^0_{\mathrm{u}}}) \to (W_{j-\Zkone, j}, \omega, \Res{F}_{W_{j-\Zkone, j}}),
  \end{align*}
  respectively which must coincide in there intersection $W \cap M_{\mathrm{r}}$, and then glue them to a morphism
  \begin{equation*}
    (\lambda_j, E_j^{-1}) \colon (W, \omega_0, \Res{q}_W) \to (W_j, \omega, \Res{F}_{W_j}).
  \end{equation*}
  For $j \in \Z_k$, let $P_j, Q_j \colon U \to M$ be Lagrangian sections of $F$ given by $P_j(c) = \lambda_j(1, c)$, $Q_j(c) = \lambda_j(c, 1)$.
  Since $\varphi_j$ is a symplectomorphism, by \cref{lem:symplecto-endo} and the definition of $\kappa$ in \cref{eq:kappa-definition}, the translation form $\tau^{P_jQ_j} = E_j^* \kappa$ is a section in $(Z^1 / 2\pi \Period)(U \cap B_\mathrm{r})$.

  Note that the images of $Q_j$ and $P_{j+\Zkone}$ lie in $W_j$.
  Then there are two smooth sections of $q$ in $\R^4 \setminus D^0_{\mathrm{s}}$ mapped to $Q_j$ and $P_{j+\Zkone}$ by $\lambda_j$ and let $\tau \in (\Omega^1 / 2\pi \Period^{(\R^4, \omega_0, q)})(\R^2)$ be the translation form between them.
  By explicitly calculating $\tau$ using \cref{eq:flow-euclidean}, it is represented by a smooth form on $\R^2$.
  So, $\tau^{Q_j P_{j+\Zkone}} = G^* \tau \in (\Omega^1 / 2\pi \Period^{(M, \omega, F)})(U \cap B_\mathrm{r})$ is represented by a smooth form on $U$, too.

  Since $\alpha_2$ is a section in $\Period(B)$ and $\kappa \in (Z^1 / 2\pi \Period)(\R^2_{\mathrm{r}})$ (hence are closed), the section
  \begin{align*}
    \sigma =2\pi \alpha_2 - \sum_{j \in \Z_k} E_j^* \kappa = 2\pi \alpha_2 - \sum_{j \in \Z_k} \tau^{P_jQ_j} = \sum_{j \in \Z_k} \tau^{Q_jP_{j+\Zkone}} \in (\Omega^1 / 2\pi \Period)(U \cap B_\mathrm{r})
  \end{align*}
  extends to a smooth and closed representative on $U$ (hence is in $(Z^1 / 2\pi \Period)(U \cap B_\mathrm{r})$).
\end{proof}

Let $\R[[T^*_0\R^2]]$ be the space of formal power series generated by the elements of a basis of $T^*_0\R^2$, or equivalently, $\R[[T^*_0U]]$ is the direct sum of symmetric tensor products of $T^*_0U$.
The \emph{Taylor series} $\Tl_0[f]$ of a smooth function $f \colon U \to \R$ at $0$ may be viewed as an element in $\R[[T^*_0U]]$.
Recall $\R[[X, Y]]$, the $\R$-algebra of formal power series in two variables $X$ and $Y$.
Let
\begin{equation*}
  \mathsf{R} = \Set{\mathsf{u}(X, Y) = \sum_{p + q \geq 1} \mathsf{u}^{(p, q)} X^p Y^q \mmid \mathsf{u}^{(p, q)} \in \R}
\end{equation*}
be the ideal of $\R[[X, Y]]$ consisting of those with no constant term.
Let $X = \der c_1$, $Y = \der c_2$ be the variables of the formal power series.
Then the diffeomorphism $E_j \colon U \to U_j$ identifies two algebras by $E_j^* \colon \R[[X, Y]] \to \R[[T^*_0U]]$.

Let $A_1$ is the action integral with $\der A_1 = \alpha_1$ and $A_1(0) = 0$.
Let $\mathsf{A}_1 = \Tl_0[A_1] \in \R[[T^*_0U]]$ and we have $(E_j^{-1})^* \mathsf{A_1} = X$ for any $j \in \Z_k$.
Since $\sigma$ defined in \cref{eq:definition-sigma} is closed, by possibly shrinking $U$ so as to be simply connected, there is a smooth $S \colon U \to \R$ such that $S(0) = 0$ and $\der S$ representats $\sigma$ in $\Omega^1(U)$.
The germ of $S$ is unique up to adding integer multiples of $2\pi A_1$.

\begin{definition} \label{def:action-Taylor-series}
  We call $\sigma \in (\Omega^1 / 2\pi \Period)(U \cap B_\mathrm{r})$ the \emph{desingularized period form}.
  We call the coset $S + 2\pi A_1 \Z$ the \emph{desingularized action integral}.
  Let $\mathsf{R}_{2\pi X} = \mathsf{R} / (2\pi X)\Z$ and let $\mathsf{S} = \Tl_0[S] + 2\pi \mathsf{A}_1 \Z \in \R[[T^*_0U]]_0 / (2\pi \mathsf{A}_1) \Z$.
  For any $m_j, j \in \Z_k$, let
  \begin{equation*}
    \mathsf{s}_j(X, Y) \defeq (E_j^{-1})^* \mathsf{S}.
  \end{equation*}
  We call $\mathsf{s}_j  \in \mathsf{R}_{2\pi X}$ the \emph{action Taylor series at $m_j$}.
\end{definition}

We construct the second set of the invariants: these invariants are Taylor series reflecting the difference between the Eliasson local charts at different singular points.

\begin{definition} \label{def:transition-cocycle}
  Let $\mathcal{R}_+ = \Set{g \colon \R^2 \to \R \mmid \partial g/\partial c_2 > 0}$ be a group with the product $(g_1 \cdot g_2)(c_1, c_2) = g_1(c_1, g_2(c_1, c_2))$ for any $g_1, g_2 \in \mathcal{R}_+$.
  Let $G_{j, \ell} = E_\ell \circ E_j^{-1}$ and $g_{j, \ell} = \proj_2 \circ G_{j, \ell} \in \mathcal{R}_+$.
  Then we have $(G_{j, \ell})(c_1, c_2) = (c_1, g_{j, \ell}(c_1, c_2))$.
  We call $\Pa{g_{j, \ell}}_{j, \ell \in \Z_k}$ the set of \emph{momentum transitions}.

  Let $\mathsf{R}_+ = \Set{\mathsf{g} \in \mathsf{R} \mmid \mathsf{g}^{(0, 1)} > 0}$ be a group with the product $(\mathsf{g}_1 \cdot \mathsf{g}_2)(X, Y) = \mathsf{g}_1(X, \mathsf{g}_2(X, Y))$ for any $\mathsf{g}_1, \mathsf{g}_2 \in \mathsf{R}_+$.
  Let $\mathsf{g}_{j, \ell} = \Tl_0[g_{j, \ell}] \in \mathsf{R}_+$.
  They satisfy the cocycle condition $\mathsf{g}_{j, \ell} \cdot \mathsf{g}_{\ell, p} = \mathsf{g}_{j, p}$.
  We call $\mathsf{g}_{j, \ell}$ the \emph{transition Taylor series} from $m_j$ to $m_\ell$.
  We call $\Pa{\mathsf{g}_{j, \ell}}_{j, \ell \in \Z_k}$ the \emph{transition cocycle}.
\end{definition}

The group structure of $\mathcal{R}_+$ is explicitly verified as follows:
for $g_1, g_2, g_2 \in \mathcal{R}_+$,
\begin{align*}
  \identity &= c_2, \\
  (g_1^{-1})(c_1, c_2) &= \proj_2 \circ G_1^{-1} (c_1, c_2)\text{ where } G_1(c_1, c_2) = (c_1, g_1(c_1, c_2)), \\
  ((g_1 \cdot g_2) \cdot g_3)(c_1, c_2) &= g_1(c_1, g_2(c_1, g_3(c_1, c_2))) = (g_1 \cdot (g_2 \cdot g_3))(c_1, c_2).
\end{align*}
The group structure of $\mathsf{R}_+$ is similarly verified.

\subsection{Moduli spaces and main theorem} \label{ssec:main-theorem}

\begin{definition} \label{def:moduli-space}
  Let $\Mffk{k}$ be the \emph{collection of integrable systems} in $\IntSysFFk{k}$ modulo isomorphisms of saturated neighborhoods of the singular fiber.
  Given two marked directed integrable systems $(M, \omega, F, (\alpha_1, \alpha_2), m_{\Zkzero})$ and $(M', \omega', F', (\alpha'_1, \alpha'_2), m'_{\Zkzero})$, an isomorphism $(\varphi, G)$ from $(M, \omega, F)$ to $(M', \omega', F')$ is said to \emph{preserve the direction and the mark} if $\varphi(m_{\Zkzero}) = m'_{\Zkzero}$ and $(G^* \alpha'_1, G^* \alpha'_2) = (\alpha_1, \alpha_2)$.
  Let $\Mffobk{k}$ be the \emph{collection of marked directed integrable systems} modulo isomorphisms of saturated neighborhoods of the singular fiber preserving the direction and the mark.
\end{definition}

Let $(M, \omega, F, (\alpha_1, \alpha_2), m_{\Zkzero})$ be a marked directed integrable system with $(M, \omega, F) \in \IntSysFFk{k}$.
Let $\Pa{m_j}_{j \in \Z_k} = \Crit(\mathcal{F})$ and suppose $\Pa{(\psi_j, E_j)}_{j \in \Z_k}$ is a singularity atlas compatible with $(\alpha_1, \alpha_2)$.
Let $\Pa{\mathsf{s}_j}_{j \in \Z_k}$ be the $k$-tuple of action Taylor series and let $\Pa{\mathsf{g}_{j, \ell}}_{j, \ell \in \Z_k}$ be the transition cocycle.
These series are constrained by the following relations:
\begin{equation} \label{eq:invariant-constraint}
\left.
\begin{aligned}
  \mathsf{s}_j = \mathsf{s}_\ell \cdot \mathsf{g}_{j, \ell} &\text{ for}~j, \ell \in \Z_k; \\
  \mathsf{g}_{j, j}(X, Y) = Y &\text{ for}~j \in \Z_k; \\
  \mathsf{g}_{\ell, p} \cdot \mathsf{g}_{j, \ell} = \mathsf{g}_{j, p} &\text{ for}~j, \ell, p \in \Z_k.
\end{aligned}
\right\}
\end{equation}

\begin{theorem} [Main Theorem] \label{thm:invariant-bijection}
  There is a bijection
  \begin{align*}
    \Phi \colon \Mffobk{k} &\to \Iffobk{k} \defeq \Set{\Pa{\Pa{\mathsf{s}_j}_{j \in \Z_k}, \Pa{\mathsf{g}_{j, \ell}}_{j, \ell \in \Z_k}} \in \mathsf{R}_{2\pi X}^k \times \mathsf{R}_+^{k^2} \mmid \cref{eq:invariant-constraint}} \\
    [(M, \omega, F, (\alpha_1, \alpha_2), m_{\Zkzero})] &\mapsto \Pa{\mathsf{s}_{\Zkzero}, \dotsc, \mathsf{s}_{\Zknegone}, \mathsf{g}_{\Zkzero, \Zkzero}, \dotsc, \mathsf{g}_{\Zkzero, \Zknegone}, \dotsc, \mathsf{g}_{\Zknegone, \Zkzero}, \dotsc, \mathsf{g}_{\Zknegone, \Zknegone}}.
  \end{align*}
\end{theorem}

Since the projection
\begin{equation*} \begin{split}
  \Iffobk{k} &\to \mathsf{R}_{2\pi X} \times \mathsf{R}_+^{k-1}, \\
    \Pa{\Pa{\mathsf{s}_j}_{j \in \Z_k}, \Pa{\mathsf{g}_{j, \ell}}_{j, \ell \in \Z_k}} &\mapsto \Pa{\mathsf{s_{\Zkzero}}, \mathsf{g}_{\Zkzero, \Zkone}, \mathsf{g}_{\Zkone, \overline{2}}, \dotsc, \mathsf{g}_{\overline{k-2}, \overline{k-1}}}.
\end{split} \end{equation*}
is one-to-one, equivalently, we have the following two statements:
\begin{itemize}
  \item \emph{Injectivity/Uniqueness}: two equivalence classes of semiglobal models of integrable systems $[(M, \omega, F, (\alpha_1, \alpha_2), m_{\Zkzero})]$ near a focus-focus fiber are isomorphic preserving the direction and the mark if and only if they have the same Taylor series invariant $\Pa{\mathsf{s_{\Zkzero}}, \mathsf{g}_{\Zkzero, \Zkone}, \mathsf{g}_{\Zkone, \overline{2}}, \dotsc, \mathsf{g}_{\overline{k-2}, \overline{k-1}}}$.
  \item \emph{Surjectivity/Existence}: given the abstract ingredient of a $k$-tuple of formal power series $\Pa{\mathsf{s_{\Zkzero}}, \mathsf{g}_{\Zkzero, \Zkone}, \mathsf{g}_{\Zkone, \overline{2}}, \dotsc, \mathsf{g}_{\overline{k-2}, \overline{k-1}}} \in \mathsf{R}_{2\pi X} \times \mathsf{R}_+^{k-1}$ there exists a unique equivalence class of semiglobal models of integrable systems $[(M, \omega, F, (\alpha_1, \alpha_2), m_{\Zkzero})]$ near a focus-focus fiber which has Taylor series invariants precisely the one we started with.
\end{itemize}
We prove \cref{thm:invariant-bijection} in the remaining sections.
In the rest of this section, we are going to prove that $\Phi$ is a well-defined map.

\begin{remark} \label{rem:compare-Vu-Ngoc}
  When $\mathcal{F}$ is single pinched, we recover the \emph{Taylor series invariant} $(S)^\infty$ from~\cite{MR1941440} by the following relation:
  \begin{equation*}
    \mathsf{s}_{\Zkzero}(X, Y) = (S)^\infty(Y, X) + \frac{\pi}{2} X.
  \end{equation*}
  The addition of $\frac{\pi}{2} X$ is due to a change in convention.
  In this paper, it is the first component of the momentum map which has periodic Hamiltonian vector fields, while in~\cite{MR1941440} it is the second one.
\end{remark}

\subsection{The invariants are well-defined} \label{ssec:invariant-uniqueness}

In this subsection, we are going to show that the output of $\Phi$ does not depend on the choice of the singularity atlas, and we also want to know how the Taylor series will change if the direction and the base point change.

Define bijections $\gamma_X$ and $\gamma_Y$ of $\Iffobk{k}$ by
\begin{equation} \begin{split} \label{eq:def-gamma-X}
  \gamma_X \Pa{\dotsc, \mathsf{s}_j, \dotsc, \mathsf{g}_{j, \ell}, \dotsc} &= \Pa{\dotsc, \mathsf{s}'_j, \dotsc, \mathsf{g}'_{j, \ell}, \dotsc}, \\
  \mathsf{s}'_j(X, Y) &= \mathsf{s}_j(-X, Y) + k \pi X, \\
  \mathsf{g}'_{j, \ell}(X, Y) &= \mathsf{g}_{j, \ell}(-X, Y);
\end{split} \end{equation}
and
\begin{equation} \begin{split} \label{eq:def-gamma-Y}
  \gamma_Y \Pa{\dotsc, \mathsf{s}_j, \dotsc, \mathsf{g}_{j, \ell}, \dotsc} &= \Pa{\dotsc, \mathsf{s}''_j, \dotsc, \mathsf{g}''_{j, \ell}, \dotsc}, \\
  \mathsf{s}''_j(X, Y) &= -\mathsf{s}_{-j}(X, -Y), \\
  \mathsf{g}''_{j, \ell}(X, Y) &= -\mathsf{g}_{-j, -\ell}(X, -Y).
\end{split} \end{equation}
Define a bijection $\theta_p$ of $\Iffobk{k}$, $p \in \Z_k$, by
\begin{equation} \begin{split} \label{eq:def-theta-p}
  \theta_p \Pa{\dotsc, \mathsf{s}_j, \dotsc, \mathsf{g}_{j, \ell}, \dotsc} = \Pa{\dotsc, \mathsf{s}_{j+p}, \dotsc, \mathsf{g}_{j+p, \ell+p}, \dotsc}.
\end{split} \end{equation}

The proof of the well-definedness lemma follows from \cite[Section 4]{MR1941440} and \cite[Lemma 4.55]{MR3765971}.

\begin{lemma} \label{lem:invariant-uniqueness}
  The map $\Phi \colon \Mffobk{k} \to \Iffobk{k}$ is well defined.
\end{lemma}

\begin{proof}
  \emph{Independence of the choice of the singularity atlas.}

  Let $[(M, \omega, F, (\alpha_1, \alpha_2), m_{\Zkzero})] \in \Mffobk{k}$ and let $\Pa{(\psi_j, E_j)}_{j \in \Z_k}$ and $\Pa{(\psi'_j, E'_j)}_{j \in \Z_k}$ be two singularity atlases compatible with $(\alpha_1, \alpha_2)$.
  Let $\Pa{\dotsc, \mathsf{s}_j, \dotsc, \mathsf{g}_{j, \ell}, \dotsc}$ and $\Pa{\dotsc, \mathsf{s}'_j, \dotsc, \mathsf{g}'_{j, \ell}, \dotsc}$ be the outputs of $\Phi$, let $\sigma$ and $\sigma'$ be the desingularized period forms, and let $\Pa{g_{j, \ell}}_{j, \ell \in \Z_k}$ and $\Pa{g'_{j, \ell}}_{j, \ell \in \Z_k}$ be the set of momentum transitions, respectively, of $\Pa{(\psi_j, E_j)}_{j \in \Z_k}$ and $\Pa{(\psi'_j, E'_j)}_{j \in \Z_k}$.
  We aim to show that $\mathsf{s}'_j =\mathsf{s}_j$ and $\mathsf{g}'_{j, \ell} = \mathsf{g}_{j, \ell}$ for $j, \ell \in \Z_k$.

  Since $G_j \defeq E'_j \circ E_j^{-1}$ is a diffeomorphism of neighborhoods of $0$ in $\R^2$ with properties that $\der (\proj_1 \circ G_j) = \der c_1$, $\frac{\partial (\proj_2 \circ G_j)}{\partial c_2} > 0$, and $q \circ \psi'_j \circ \psi_j^{-1} = G_j \circ q$, by \cref{lem:local-flat} we have $G_j(c_1, c_2) = (c_1, c_2 + \inftes)$.
  Then \cref{lem:kappa-lemma}, we have $G_j^* \kappa = \kappa + \inftes \der c_1 + \inftes \der c_2$, and we have
  \begin{align*}
    \mathsf{g}'_{j, \ell} &= \Tl_0[\proj_2 \circ G_\ell \circ G_{j, \ell} \circ G_j^{-1}] = \Tl_0[\proj_2 \circ G_{j, \ell}] = \mathsf{g}_{j, \ell}, \\
    \begin{split}
      \sigma' - \sigma &= \Pa{2\pi \alpha_2 - \sum_{j \in \Z_k} (E'_j)^* \kappa} - \Pa{2\pi \alpha_2 - \sum_{j \in \Z_k} E_j^* \kappa} \\
      &= \sum_{j \in \Z_k} E_j^* \Pa{\kappa - G_j^* \kappa} = \inftes \der c_1 + \inftes \der c_2 + 2\pi \alpha_2 \Z,
    \end{split}
  \end{align*}
  which implies that $\mathsf{S}' = \mathsf{S}, \mathsf{s}'_j =\mathsf{s}_j$ for $j \in \Z_k$.
  \bigskip

  \emph{Independence of the choice of the representative.}

  Let $(M, \omega, F, (\alpha_1, \alpha_2), m_{\Zkzero})$ and $(M', \omega', F', (\alpha'_1, \alpha'_2), m'_{\Zkzero})$ be two marked directed integrable systems related by an isomorphism $(\varphi, G)$ between saturated neighborhoods of $\mathcal{F}$ and $\mathcal{F}'$ preserving the direction and the mark.
  That is to say that $\varphi(m_{\Zkzero}) = m'_{\Zkzero}$ and $(G^* \alpha'_1, G^* \alpha'_2) = (\alpha_1, \alpha_2)$.
  Let
  \begin{align*}
    \Pa{\dotsc, \mathsf{s}_j, \dotsc, \mathsf{g}_{j, \ell}, \dotsc} &= \Phi\Big([(M, \omega, F, (\alpha_1, \alpha_2), m_{\Zkzero})]\Big), \\
    \Pa{\dotsc, \mathsf{s}'_j, \dotsc, \mathsf{g}'_{j, \ell}, \dotsc} &= \Phi\Big([(M', \omega', F', (\alpha'_1, \alpha'_2), m'_{\Zkzero})]\Big).
  \end{align*}
  According to the first part, the above two lines do not depend on the choices of compatible singularity atlases and we aim to show that the above two lines are equal.
  Let $\Pa{(\psi_j, E_j)}_{j \in \Z_k}$ be a singularity atlas of $(M, \omega, F)$ compatible with $(\alpha_1, \alpha_2)$ and then $\Pa{(\psi'_j, E'_j)}_{j \in \Z_k} \defeq \Pa{(\psi_j \circ \varphi^{-1}, E_j \circ G^{-1})}_{j \in \Z_k}$ is a singularity atlas of $(M', \omega', F')$ compatible with $(\alpha'_1, \alpha'_2)$.
  Let $\sigma$ and $\sigma'$ be the desingularized period forms, and then
  \begin{align*}
    G^* \sigma' &= 2\pi G^* \alpha'_2 - \sum_{j \in \Z_k} (E'_j \circ G)^* \kappa = 2\pi \alpha_2 - \sum_{j \in \Z_k} E_j^* \kappa = \sigma, \\
    E'_\ell \circ (E'_j)^{-1} &= (E_\ell \circ G^{-1}) \circ (E_j \circ G^{-1})^{-1} = E_\ell \circ E_j^{-1}.
  \end{align*}
  Hence we have $G \circ \mathsf{S}' = \mathsf{S}$, $\mathsf{s}'_j = \mathsf{s}_j$, and $\mathsf{g}'_{j, \ell} = \mathsf{g}_{j, \ell}$ for $j, \ell \in \Z_k$.
\end{proof}

\begin{lemma} \label{lem:invariant-group-action}
  The map $\Phi \colon \Mffobk{k} \to \Iffobk{k}$ satisfies the relations:
  \begin{equation*}
    \left.
    \begin{aligned}
      \Phi\Big([(M, \omega, F, (-\alpha_1, \alpha_2), m_{\Zkzero})]\Big) &= \gamma_X\Pa{\Phi\Big([(M, \omega, F, (\alpha_1, \alpha_2), m_{\Zkzero})]\Big)}, \\
      \Phi\Big([(M, \omega, F, (\alpha_1, -\alpha_2), m_{\Zkzero})]\Big) &= \gamma_Y\Pa{\Phi\Big([(M, \omega, F, (\alpha_1, \alpha_2), m_{\Zkzero})]\Big)}, \\
      \Phi\Big([(M, \omega, F, (\alpha_1, \alpha_2), m_p)]\Big) &= \theta_p\Pa{\Phi\Big([(M, \omega, F, (\alpha_1, \alpha_2), m_{\Zkzero})]\Big)}, \,\text{for }p \in \Z_k.
    \end{aligned}
    \right.
  \end{equation*}
\end{lemma}

\begin{proof}
  Let $[(M, \omega, F, (\alpha_1, \alpha_2), m_{\Zkzero})] \in \Mffobk{k}$ and let $\Pa{(\psi_j, E_j)}_{j \in \Z_k}$ be a singularity atlas compatible with $(\alpha_1, \alpha_2)$.
  Let $\Pa{(\psi'_j, E'_j)}_{j \in \Z_k}$ be another singularity atlas and let $m'_{\Zkzero}$ be another mark.
  We may need to reorder $\Crit(\mathcal{F})$ to $\Pa{m'_j}_{j \in \Z_k}$ by \cref{lem:existence-compatible-singularity-atlas} so that $(\psi'_j, E'_j)$ is a chart near $m'_j$ for $j \in \Z_k$ and $\Pa{(\psi'_j, E'_j)}_{j \in \Z_k}$ is compatible with some direction $(\alpha'_1, \alpha'_2) \in \Dir(M, \omega, F)$.

  Let
  \begin{align*}
    \Pa{\dotsc, \mathsf{s}_j, \dotsc, \mathsf{g}_{j, \ell}, \dotsc} &= \Phi\Big([(M, \omega, F, (\alpha_1, \alpha_2), m_{\Zkzero})]\Big), \\
    \Pa{\dotsc, \mathsf{s}'_j, \dotsc, \mathsf{g}'_{j, \ell}, \dotsc} &= \Phi\Big([(M, \omega, F, (\alpha'_1, \alpha'_2), m'_{\Zkzero})]\Big).
  \end{align*}
  Let $\sigma, \sigma'$ be the desingularized period forms, and $\Pa{g_{j, \ell}}_{j, \ell \in \Z_k}, \Pa{g'_{j, \ell}}_{j, \ell \in \Z_k}$ respectively be the set of momentum transitions of $\Pa{(\psi_j, E_j)}_{j \in \Z_k}, \Pa{(\psi'_j, E'_j)}_{j \in \Z_k}$.
  \bigskip

  \emph{Case 1:}
  If $(\alpha'_1, \alpha'_2) = (-\alpha_1, \alpha_2)$ and $m'_{\Zkzero} = m_{\Zkzero}$, then $m'_j = m_j$ for $j \in \Z_k$.
  We have $G_X^* \kappa = \kappa + \pi \der c_1$ (recall $G_X$ defined in \cref{eq:local-model-diffeomorphism}).
  Since $\Pa{(\psi'_j, G_X \circ E_j)}_{j \in \Z_k}$ is compatible with $(\alpha'_1, \alpha'_2)$, by \cref{lem:invariant-uniqueness}, it is sufficient to assume that $E'_j = G_X \circ E_j$.
  Then
  \begin{equation*} \begin{split}
    \sigma' &= 2\pi \alpha'_2 - \sum_{j \in \Z_k} (E'_j)^* \kappa = 2\pi \alpha_2 - \sum_{j \in \Z_k} E_j^* (\kappa + \pi \der c_1) \\
    &= \sigma - \pi \sum_{j \in \Z_k} E_j^* \der c_1 = \sigma - k \pi \der c_1,
  \end{split} \end{equation*}
  and $g'_{j, \ell}(c) = \proj_2 \circ G_X \circ G_{j, \ell} \circ G_X^{-1}(c) = g_{j, \ell}(-\cj{c})$.
  In this case,
  \begin{align*}
    \mathsf{S}' &= \mathsf{S} - k \pi[c_1], \\
    \mathsf{s}'_j(X, Y) &= \mathsf{s}_j(-X, Y) + k \pi X, \\
    \mathsf{g}'_{j, \ell}(X, Y) &= \mathsf{g}_{j, \ell}(-X, Y).
  \end{align*}
  \bigskip

  \emph{Case 2:}
  If $(\alpha'_1, \alpha'_2) = (\alpha_1, -\alpha_2)$ and $m'_{\Zkzero} = m_{\Zkzero}$, then $m'_j = m_{-j}$ for $j \in \Z_k$ (since the direction of $\mathcal{X}_{E_{\Zkzero}^* \der c_2}$ is reversed).
  We have $G_Y^* \kappa = -\kappa$ (recall $G_Y$ defined in \cref{eq:local-model-diffeomorphism}).
  By \cref{lem:invariant-uniqueness}, it is sufficient to assume that $E'_j = G_Y \circ E_{-j}$.
  Then 
  \begin{equation*} \begin{split}
    \sigma' = 2\pi \alpha'_2 - \sum_{j \in \Z_k} (E'_j)^* \kappa = -2\pi \alpha_2 + \sum_{j \in \Z_k} E_{-j}^* \kappa = -\sigma,
  \end{split} \end{equation*}
  and $g'_{j, \ell}(c) = \proj_2 \circ G_Y \circ G_{-j, -\ell} \circ G_Y^{-1}(c) = -g_{-j, -\ell}(\cj{c})$.
  In this case,
  \begin{align*}
    \mathsf{S}' &= -\mathsf{S}, \\
    \mathsf{s}'_j(X, Y) &= -\mathsf{s}_{-j}(X, -Y), \\
    \mathsf{g}'_{j, \ell}(X, Y) &= -\mathsf{g}_{-j, -\ell}(X, -Y).
  \end{align*}
  \bigskip

  \emph{Case 3:}
  If $(\alpha'_1, \alpha'_2) = (\alpha_1, \alpha_2)$ and $m'_{\Zkzero} = m_p$ for some $p \in \Z_k$, then $m'_j = m_{j+p}$ for $j \in \Z_k$.
  By \cref{lem:invariant-uniqueness}, it is sufficient to assume that $E'_j = E_{j+p}$.
  Then $\sigma' = \sigma$ and $g'_{j, \ell} = g_{j+p, \ell+p}$.
  Hence $\mathsf{S}' = \mathsf{S}, \mathsf{s}'_j = \mathsf{s}_{j+p}, \mathsf{g}'_{j, \ell} = \mathsf{g}_{j+p, \ell+p}$.
\end{proof}

\subsection{Corollary of the main theorem} \label{ssec:corollary-main-theorem}

The bijections $\gamma_X$, $\gamma_Y$, and $\theta_p$ are subject to the relations
\begin{align*}
  \gamma_X^2 = \gamma_Y^2 = \theta_p^p = (\gamma_Y \circ \theta_p)^2 = \identity.
\end{align*}
So they generate a $(\Z_2 \times D_k)$-action on $\Iffobk{k}$; $\gamma_X$ generates $\Z_2$, $\gamma_Y$ and $\theta_p$ generate $D_k$.
We have

\begin{corollary} [Corollary of \cref{thm:invariant-bijection}] \label{cor:invariant-bijection-quotient}
  There is a bijection
  \begin{align*}
    \wt{\Phi} \colon \Mffk{k} &\to \Iffk{k} \defeq \Iffobk{k} / (\Z_2 \times D_k) \\
    [(M, \omega, F)] &\mapsto \left[\Phi\Pa{[(M, \omega, F, (\alpha_1, \alpha_2), m_{\Zkzero})]}\right]
  \end{align*}
  where $(\alpha_1, \alpha_2) \in \Dir(M, \omega, F)$, $m_{\Zkzero}$ is a singular point of $F$, and the $(\Z_2 \times D_k)$-action is generated by $\gamma_X$, $\gamma_Y$, and $\theta_p$.
\end{corollary}

\begin{remark} \label{rem:four-fold-action-by-direction}
  As pointed out in~\cite{MR3765971} the Taylor series invariant in the case that the singular fiber contains exactly one critical point of focus-focus is defined up to a $(\Z_2 \times \Z_2)$-action, which accounts for the choices of Eliasson local charts in its construction.
  It becomes unique in the presence of a global $\mathbb{S}^1$-action (i.e., semitoric systems) provided one assumes everywhere that the Eliasson local charts preserve the $\mathbb{S}^1$-action and the $\mathbb{R}^2$-direction.
  In \cref{cor:invariant-bijection-quotient}, we have the $(\Z_2 \times D_k)$-action instead.
  When $k = 1$, $(\Z_2 \times D_k) \simeq (\Z_2 \times \Z_2)$.
\end{remark}

\begin{remark} \label{rem:compare-Izosimov}
  We can recover the smooth invariant in~\cite[Theorem 3.7]{MR4057723} from our symplectic invariant $(\Z_2 \times D_k) \cdot \Pa{\Pa{\mathsf{s}_j}_{j \in \Z_k}, \Pa{\mathsf{g}_{j, \ell}}_{j, \ell \in \Z_k}}$.
  To any $\mathsf{w} \in \mathsf{R}_+$ we assign a complex formal series $\mathsf{w}_\C(Z, \overline{Z}) = X + \imag \mathsf{w}(X, Y)$ by setting $Z = X + \imag Y$ and then let 
  \begin{equation*}
    \mathsf{LR}_+ = \Set{\mathsf{w} \in \mathsf{R}_+ \mmid \mathsf{w}_\C(Z, \overline{Z}) = \sum_{r = 1}^\infty \sum_{s = 0}^\infty \mathsf{w}_\C^{(r, \overline{s})} Z^r \overline{Z}\vphantom{Z}^s, \mathsf{w}_\C^{(r, \overline{s})} \in \C}.
  \end{equation*}
  Then $\mathsf{LR}_+$ is the space of the Taylor series of the second components of orientation-preserving liftable diffeomorphisms characterized by~\cite[Theorem 3.4]{MR4057723} that preserves the first coordinate.
  Consider the action of $\mathsf{LR}_+^{k-1}$ on $\mathsf{R}_+^{k^2}$ by
  \begin{equation*}
    \Pa{\mathsf{w}_j}_{j \in \Z_k} \cdot \Pa{\mathsf{g}_{j, \ell}}_{j, \ell \in \Z_k} = \Pa{\mathsf{w}_\ell \cdot \mathsf{g}_{j, \ell} \cdot \mathsf{w}_j^{-1}}_{j, \ell \in \Z_k}
  \end{equation*}
  for $\Pa{\mathsf{w}_j}_{j \in \Z_k \setminus \scriptstyle\Set{\Zkzero}} \in \mathsf{LR}_+^{k-1}$ and $\mathsf{w}_{\Zkzero}(X, Y) = Y$.
  The smooth invariant is equivalent to the orbit of
  \begin{align*}
    (\mathsf{LR}_+^{k-1} \rtimes (\Z_2 \times D_k)) \cdot \Pa{\mathsf{g}_{j, \ell}}_{j, \ell \in \Z_k}.
  \end{align*}
  The $\mathrm{C}^1$-invariant is equivalent to the tuple of numbers given by 
  \begin{equation*}
    \mu_j = \frac{(\mathsf{g}_{\Zkzero, j})_\C^{(0, \overline{1})}}{\overline{(\mathsf{g}_{\Zkzero, j})_\C}^{(1, \overline{0})}} = \frac{1 - \displaystyle\Pa{\mathsf{g}_{\Zkzero, j}^{(1, 0)}}^2 - \displaystyle\Pa{\mathsf{g}_{\Zkzero, j}^{(0, 1)}}^2 + 2\imag \mathsf{g}_{\Zkzero, j}^{(1, 0)}}{\displaystyle\Pa{\mathsf{g}_{\Zkzero, j}^{(1, 0)}}^2 + \displaystyle\Pa{1 + \mathsf{g}_{\Zkzero, j}^{(0, 1)}}^2}
  \end{equation*}
  for $j \in \Z_k \setminus \Set{\Zkzero}$, up to simultaneously multiplying a complex number of unit norm or taking complex conjugations.
  Note that there are convention changes to the definitions so we said ``is equivalent to'' instead of ``is''.
  Note also that the realization theorem~\cite[Theorem 3.7]{MR4057723} ensured that every smooth focus-focus singularity is diffeomorphic to a symplectic one, so we have quotient maps from symplectic invariants to smooth and $\mathrm{C}^1$-invariants.
\end{remark}

\section{Proof of uniqueness} \label{sec:invariant-injectivity}

The goal of this section is to show \cref{lem:invariant-injectivity}, the injectivity claim of \cref{thm:invariant-bijection}.
Let $[(M, \omega, F, (\alpha_1, \alpha_2), m_{\Zkzero})], [(M', \omega', F', (\alpha'_1, \alpha'_2), m'_{\Zkzero})] \in \Mffobk{k}$ such that
\begin{equation*}
  \Phi\Big([(M, \omega, F, (\alpha_1, \alpha_2), m_{\Zkzero})]\Big) = \Phi\Big([(M', \omega', F', (\alpha'_1, \alpha'_2), m'_{\Zkzero})]\Big).
\end{equation*}
Let $B = F(M)$, $B' = F'(M')$.
Let $\mathcal{F} = F^{-1}(0)$ and $\mathcal{F}' = (F')^{-1}(0)$ be the singular fibers.
Let $\Pa{m_j}_{j \in \Z_k} = \Crit(\mathcal{F})$ and $\Pa{m'_j}_{j \in \Z_k} = \Crit(\mathcal{F}')$.
We aim to show that, there are $U \in \mathcal{N}(B, 0)$ and $U' \in \mathcal{N}(B', 0)$, a symplectomorphism $\varphi \colon (F^{-1}(U), \omega, \mathcal{F}) \to ((F')^{-1}(U'), \omega', \mathcal{F}')$ lifting a diffeomorphism $G \colon U \to U'$ such that $(G^* \alpha'_1, G^* \alpha'_2) = (\alpha_1, \alpha_2)$ and $\varphi(m_j) = m'_j$.

The following lemma is analogous to V{\~u} Ng{\d{o}}c {\cite[Lemma 5.1]{MR1941440}}.

\begin{lemma} \label{lem:change-A2}
  Suppose $\beta, \beta' \in (\Omega^1 / 2\pi \Period^{(\R^4, \omega_0, q)})(U \cap \R^2_\mathrm{r})$ for some $U \in \mathcal{N}(\R^2, 0)$ such that $\beta$ is of the form $2\pi \beta = \tau + \sum_{j \in \Z_k} E_j^* \kappa$ for some $\tau \in \Omega^1(U)$ and diffeomorphisms $E_j \colon U \to E_j(U) \subseteq \R^2$, and $\beta' - \beta$ is a closed and flat.
  Then there is a diffeomorphism $G \colon \R^2 \to \R^2$ isotopic to the identity and a $U' \in \mathcal{N}(\R^2, 0)$ such that $G(c_1, c_2) = (c_1, c_2 + \inftes)$ and $G^* \beta' = \beta$ on $U'$.
\end{lemma}

\begin{proof}
  Let $\rho = \beta' - \beta \in Z^1(\R^2)$.
  Throughout the proof $t$ is a variable in $[0, 1]$.
  Let $\beta_t = \beta + t \rho$, and let $R \in \inftes$ be such that $\der R = \rho$.
  Define a family of functions $h_t \colon U \to \R$ as
  \begin{equation*}
    h_t = \Braket{\beta_t, \frac{\partial}{\partial c_2}} = \Braket{\frac{\tau}{2\pi} + t \rho, \frac{\partial}{\partial c_2}} - \sum_{j \in \Z_k} \frac{\partial (\proj_2 \circ E_j)}{\partial c_2} \frac{\ln \abs{E_j}}{2\pi}.
  \end{equation*}
  Since $\frac{\partial (\proj_2 \circ E_j)}{\partial c_2}(0) > 0$ for any $j \in \Z_k$, we have $h_t(c) \to \infty$ as $c \to 0$.
  Note that for any multi-index $j$ the partial derivative $\partial^j h_t$ is a polynomial of $\abs{E_j}^{-1}$ and $\ln \abs{E_j}$, $j \in \Z_k$ with coefficients as smooth functions, divided by the $\abs{j}$-th power of $h_t$.
  Thus $1/h_t$ satisfies \cref{eq:tempered} (in place of $h$).
  Since $R$ is flat, by \cref{lem:flat-times-tempered} and using a bump function, for any $U'' \in \mathcal{N}(\R^2, 0)$ with $\overline{U''} \subset U$, the family of functions $f_t \defeq -\Res{R}_{U''}/h_t$ on $U'' \cap \R^2_{\mathrm{r}}$ has a smooth extension $\tilde{f}_t \colon \R^2 \to \R$ that is flat at the origin and has compact support.
  Take $G = G_1$ as $G_t$ to be the flow of $Y_t = \tilde{f}_t \frac{\partial}{\partial c_2}$ on $\R^2$.
  Then
  \begin{equation*}
    \frac{\der}{\der t} (G_t^* \beta_t) = G_t^* (\der \langle\beta_t, Y_t\rangle + \rho) = G_t^* (\der(\tilde{f}_t\langle\beta_t, \textstyle\frac{\partial}{\partial c_2}\rangle) + \rho) = G_t^* (-\der R + \rho) = 0
  \end{equation*}
  on $G_t^{-1}(U'')$.
  Hence $G_t^* \beta_t = \beta$ on $U'$, the intersection of $G_t^{-1}(U'')$ for $t \in [0, 1]$ and by the construction $G_t(c_1, c_2) = (c_1, c_2 + \inftes)$.
\end{proof}

\begin{lemma} \label{lem:invariant-injectivity}
  The map $\Phi \colon \Mffobk{k} \to \Iffobk{k}$ is injective.
\end{lemma}

\begin{proof}
  \emph{Initialization:}
  Let $\Pa{(\psi_j \colon V_j \to \psi_j(V_j), E_j \colon U \to U_j)}_{j \in \Z_k}$ be a singularity atlas of $(M, \omega, F)$ compatible with $(\alpha_1, \alpha_2)$, and let $\Pa{(\psi'_j \colon V'_j \to \psi'_j(V'_j), E'_j \colon U' \to U'_j)}_{j \in \Z_k}$ be a singularity atlas of $(M', \omega', F')$ compatible with $(\alpha'_1, \alpha'_2)$.
  Suppose the two marked directed integrable systems share the same set of invariants $\Pa{\Pa{\mathsf{s}_j}_{j \in \Z_k}, \Pa{\mathsf{g}_{j, \ell}}_{j, \ell \in \Z_k}}$.
  For $j, \ell \in \Z_k$ recall that $G_{j, \ell} = E_\ell \circ E_j^{-1}$ and let $G'_{j, \ell} = E'_\ell \circ (E'_j)^{-1}$; then both 
  \begin{align*}
    ((E'_j)^{-1})^* \sigma' - (E_j^{-1})^* \sigma, \qquad
    \proj_2 \circ (G'_{j, \ell} - G_{j, \ell})
  \end{align*}
  are flat.
  We will construct a new system over $U$ isomorphic to $(M', \omega', F')$ over $U'$ with a singularity atlas of the system compatible with $(\alpha_1, \alpha_2)$.

  Calculating a smooth representative of the difference of two singular forms
  \begin{equation*} \begin{split}
    ((E'_{\Zkzero})^{-1})^* \alpha'_2 - (E_{\Zkzero}^{-1})^* \alpha_2 &= \frac{1}{2\pi} \Pa{((E'_{\Zkzero})^{-1})^*\sigma' - \sum_{j \in \Z_k} (G'_{\Zkzero, j})^* \kappa} - \Pa{(E_{\Zkzero}^{-1})^* \sigma - \sum_{j \in \Z_k} G_{\Zkzero, j}^* \kappa} \\
    &= \Pa{((E'_j)^{-1})^* \sigma' - (E_j^{-1})^* \sigma} - \sum_{j \in \Z_k} E_j^* \Pa{(G'_{\Zkzero, j})^* \kappa - G_{\Zkzero, j}^* \kappa} \\
    &= \inftes \der c_1 + \inftes \der c_2.
  \end{split} \end{equation*}
  We have used the fact that, by \cref{lem:kappa-lemma},
  \begin{equation*}
    (G'_{\Zkzero, j})^* \kappa - G_{\Zkzero, j}^* \kappa = G_{\Zkzero, j}^* \Pa{(G_{\Zkzero, j}^{-1} \circ G'_{\Zkzero, j})^* \kappa - \kappa} = \inftes \der c_1 + \inftes \der c_2.
  \end{equation*}
  By \cref{lem:change-A2} and shrinking the neighborhood $U$, and then $U_j = E_j(U)$ and $U'_j = E'_j(U)$ for $j \in \Z_k$ are shrunk accordingly, there is diffeomorphism $G_1$ of $\R^2$ such that
  \begin{align*}
    G_1^* ((E'_{\Zkzero})^{-1})^*  \alpha'_1 =(E_{\Zkzero}^{-1})^* \alpha_1, \quad G_1^* ((E'_{\Zkzero})^{-1})^*  \alpha'_2 = (E_{\Zkzero}^{-1})^* \alpha_2
  \end{align*}
  on $U_{\Zkzero}$.
  Let $G = (E'_{\Zkzero})^{-1} \circ G_1 \circ E_{\Zkzero} \colon U \to U' = (E'_{\Zkzero})^{-1}(U'_{\Zkzero})$ and then we have $(G^* \alpha'_1, G^* \alpha'_2) = (\alpha_1, \alpha_2)$.
  We calculate the Taylor series of the diffeomorphism $G'_j = E'_j \circ G \circ E_j^{-1} \colon U_j \to U'_j$,
  \begin{equation*} \begin{split}
    \Tl_0[G'_j] &\defeq (\Tl_0[\proj_1 \circ G'_j], \Tl_0[\proj_1 \circ G'_j]) \\
    &= \Tl_0[E'_j \circ G \circ E_j^{-1}] \\
    &= \Tl_0[E'_j \circ (E'_{\Zkzero})^{-1} \circ G_1 \circ E_{\Zkzero} \circ E_j^{-1}] \\
    &= \Tl_0[G'_{\Zkzero, j} \circ G_1 \circ G_{\Zkzero, j}^{-1}] \\
    &= \Tl_0[G_1] = (X, Y).
  \end{split} \end{equation*}
  Then $G'_j(c_1, c_2) = (c_1, c_2 + \inftes)$.
  Define symplectomorphisms
  \begin{align*}
    \wt{\varphi}_{G'_j} &\colon (q^{-1}(U_j), \omega_0) \to (q^{-1}(U'_j), \omega_0)
  \end{align*}
  as in \cref{lem:varphi_G}, lifting $G'_j$.
  Therefore, $(\wt{\varphi}_{G'_j}^{-1} \circ \psi'_j, E_j)$ is an Eliasson local chart at $m'_j$ for $j \in \Z_k$ and $\Pa{(\wt{\varphi}_{G'_j}^{-1} \circ \psi'_j, E_j)}_{j \in \Z_k}$ is a singularity atlas of $(M', \omega', G^{-1} \circ F')$, both compatible with $(\alpha_1, \alpha_2)$.

  By replacing $(M', \omega', F', (\alpha'_1, \alpha'_2), m'_{\Zkzero})$ with $(M', \omega', G^{-1} \circ F', (\alpha_1, \alpha_2), m'_{\Zkzero})$ and $\Pa{(\psi'_j, E'_j)}_{j \in \Z_k}$ with $\Pa{(\wt{\varphi}_{G'_j}^{-1} \circ \psi'_j, E_j)}_{j \in \Z_k}$ if necessary, we assume later, without loss of generality, that $(\alpha'_1, \alpha'_2) = (\alpha_1, \alpha_2)$ and $E'_\ell \circ (E'_j)^{-1} = E_\ell \circ E_j^{-1}$ for $j, \ell \in \Z_k$.
  \bigskip

  \emph{Construction of the semiglobal isomorphism:}
  After the initialization we have marked directed systems $(M, \omega, F, (\alpha_1, \alpha_2), m_{\Zkzero})$ with a singularity atlas $\Pa{(\psi_j \colon V_j \to \psi_j(V_j), E_j \colon U \to U_j)}_{j \in \Z_k}$, and $(M', \omega', F', (\alpha_1, \alpha_2), m'_{\Zkzero})$ with a singularity atlas $\Pa{(\psi'_j \colon V'_j \to \psi'_j(V'_j), E'_j \colon U' \to U'_j)}_{j \in \Z_k}$ such that $E'_\ell \circ (E'_j)^{-1} = E_\ell \circ E_j^{-1}$ for $j, \ell \in \Z_k$.
  We aim to find a symplectomorphism $\varphi \colon (W, \omega) \to (W', \omega')$ with $F' \circ \varphi = F$, and $\varphi(m_j) = m'_j$ where $W = F^{-1}(U) \in \mathcal{N}_F(M, \mathcal{F})$ and $W' = (F')^{-1}(U') \in \mathcal{N}_{F'}(M', \mathcal{F}')$.

  We construct $\varphi$ by induction as follows.
  Define a symplectomorphism $\varphi_{\Zkzero} = (\psi'_{\Zkzero})^{-1} \circ \psi_{\Zkzero} \colon (V_{\Zkzero}, \omega) \to (V'_{\Zkzero}, \omega')$.
  Recall the definition of $W_j$ in \cref{eq:def-Mj}; $M'_j \subset M'$ is defined analogously.
  Analogous to the proof of \cref{lem:definition-sigma}, we can extend $\varphi_{\Zkzero}$ to a symplectomorphism $\wt{\varphi}_{\Zkzero} \colon (W_{\Zkzero}, \omega) \to (W'_{\Zkzero}, \omega')$.

  For $j \in \Z_k \setminus \Set{\Zknegone}$, suppose we have defined the symplectomorphism $\wt{\varphi}_j \colon (W_j, \omega) \to (W'_j, \omega')$, and we want to define $\wt{\varphi}_{j+\Zkone} \colon (W_{j+\Zkone}, \omega) \to (W'_{j+\Zkone}, \omega')$.
  Let $\mu_{j+\Zkone}$ be a symplectomorphism determined by the following commutative diagram:
  \begin{equation*}
    \xymatrix{
      (V_j, \omega) \ar[r]^{\varphi_j} \ar[d]_{\psi_{j+\Zkone}^{-1} \circ \psi_j} & (V'_j, \omega') \ar[d]^{(\psi'_{j+\Zkone})^{-1} \circ \psi'_j} \\
      (V_{j+\Zkone}, \omega) \ar[r]_{\mu_{j+\Zkone}} & (V'_{j+\Zkone}, \omega')
    }.
  \end{equation*}

  By \cref{lem:symplecto-extension}, we extend $\mu_{j+\Zkone}$ to $\wt{\mu}_{j+\Zkone} \colon (W_{j+\Zkone}, \omega) \to (W'_{j+\Zkone}, \omega')$.
  Recall $W_{j, j+\Zkone} = W_j \cap W_{j+\Zkone}$ and let $M'_{j, j+\Zkone} = M'_j \cap M'_{j+\Zkone}$.
  Define $\mu_{j, j+\Zkone}$ such that the diagram
  \begin{equation*}
    \xymatrix{
      (W_{j, j+\Zkone}, \omega) \ar[r]^{\wt{\varphi}_j} \ar[dr]_{\wt{\mu}_{j+\Zkone}} & (W'_{j, j+\Zkone}, \omega') \ar[d]^{\mu_{j, j+\Zkone}} \\
      & (W'_{j, j+\Zkone}, \omega')
    }
  \end{equation*}
  commutes.

  Note that $\wt{\mu}_{j+\Zkone}(m_{j+\Zkone}) = m'_{j+\Zkone}$ and $\wt{\varphi}_j(x) \to m'_{j+\Zkone}$ as $x \to m_{j+\Zkone}$ in $M$, so we have $\mu_{j, j+\Zkone}(m'_{j+\Zkone}) = m'_{j+\Zkone}$.
  As a fiberwise translation by $\tau'_{j, j+\Zkone} \in \Omega^1(U')$, in the sense that $\mu_{j, j+\Zkone} = \Res{\Psi_{\tau'_{j, j+\Zkone}}}_{W'_{j, j+\Zkone}}$, we could extend $\mu_{j, j+\Zkone}$ to a symplectomorphism $\wt{\mu}_{j, j+\Zkone} = \Psi_{\tau_{j, j+\Zkone}}$ of $(W', \omega')$.
  Now let
  \begin{equation*} \begin{split}
    \wt{\varphi}_{j+\Zkone} = \wt{\mu}_{j+\Zkone}^{-1} \circ \Res{\wt{\mu}_{j, j+\Zkone}}_{W_{j+\Zkone}} \colon (W_{j+\Zkone}, \omega) &\to (W'_{j+\Zkone}, \omega'),
  \end{split} \end{equation*}
  and then $\wt{\varphi}_{j+\Zkone} = \wt{\varphi}_j$ in their common domain $W_{j, j+\Zkone}$.

  For $\varphi_{\Zknegone}$ and $\varphi_{\Zkzero}$, they coincide on regular values of $F$ near $\mathcal{F}$, so by continuity, they must coincide on their common domain $W_{\Zknegone, \Zkzero}$.
  Hence, we can glue $\wt{\varphi}_j$, $j \in \Z_k$ to get a symplectomorphism $\varphi \colon (W, \omega) \to (W', \omega')$ with the commuting diagram:
  \begin{equation*}
    \xymatrix{
      (W, \omega) \ar[rr]^{\varphi} \ar[dr]_{F} & & (W', \omega') \ar[dl]^{F'} \\
      & U &
     }.
  \end{equation*}
\end{proof}

\section{Proof of existence} \label{sec:invariant-surjectivity}

The goal of this section is to show \cref{lem:invariant-surjectivity}, the surjectivity claim of \cref{thm:invariant-bijection}.
Let
\begin{equation*}
  \Pa{\dotsc, \mathsf{v}_j, \dotsc, \mathsf{w}_{j, \ell}, \dotsc} \in \Iffobk{k}.
\end{equation*}
We aim to show that there is $(M, \omega, F, (\alpha_1, \alpha_2), m_{\Zkzero})$ such that
\begin{equation} \label{eq:invariants-coincide}
  \Phi\Big([(M, \omega, F, (\alpha_1, \alpha_2), m_{\Zkzero})]\Big) = \Pa{\dotsc, \mathsf{v}_j, \dotsc, \mathsf{w}_{j, \ell}, \dotsc}.
\end{equation}

The local structures of the integrable system $(M, \omega, F)$ near the singular points $m_j$ are isomorphic to the local normal form in \cref{ssec:normal-form}.
The isomorphism can be extended to a complete neighborhood $W_j$ of $m_j$.
We use the symplectic gluing technique similar to~\cite[Section~3]{MR2784664} to construct $(M, \omega, F)$.

By Borel's lemma, there is an open neighborhood $U$ of the origin in $\R^2$ and smooth maps $s_{\Zkzero} \colon U \to \R$ and $G_{\Zkzero, j} \colon U \to \R^2$ such that $\Tl_0[s_{\Zkzero}] = \mathsf{v}_{\Zkzero}$ and $\Tl_0[G_{\Zkzero, j}] = (X, \mathsf{w}_{\Zkzero, j})$.
We choose $G_{\Zkzero, \Zkzero} = \identity$.
Let $\mathsf{w}_{\Zkzero, j}^Y > 0$ be the $Y$-coefficient of $\mathsf{w}_{\Zkzero, j}$, and $\mathsf{v}_{\Zkzero}^Y \in \R$ the $Y$-coefficient of $\mathsf{v}_{\Zkzero}$.
For $\delta > 0$ sufficiently small, there is an open neighborhood $U_{\Zkzero} \subset U$ of $0$, such that for any $j \in \Z_k$, if we let $U_j = G_{\Zkzero, j}(U_{\Zkzero})$, then for any $c \in U_j$ we have
\begin{align} \label{eq:estimate-delta}
  \abs{c} < \delta < 1, \qquad \frac{\partial (\proj_2 \circ G_{\Zkzero, j})}{\partial c_2}(c) > \mathsf{w}_{\Zkzero, j}^Y - \delta > 0, \qquad \frac{\partial s_{\Zkzero}}{\partial c_2}(c) > \mathsf{v}_{\Zkzero}^Y - \delta > 0.
\end{align}
For $j \in \Z_k$, let $W_j = q^{-1}(U_j)$, $W_{j, \mathrm{nu}} = W_j \setminus D^0_{\mathrm{u}}$, $W_{j, \mathrm{ns}} = W_j \setminus D^0_{\mathrm{s}}$, $W_{j, \mathrm{r}} = W_j \cap \R^4_{\mathrm{r}}$, $U_{j, \mathrm{r}} = U_j \cap \R^2_{\mathrm{r}}$.
Let $kW = \coprod_{j \in \Z_k} W_j$ and $kW_{\mathrm{r}} = \coprod_{j \in \Z_k} W_{j, \mathrm{r}}$.
Let $kD = \coprod_{j \in \Z_k} D_j \subset kW$ where
\begin{align*}
  D_{\Zkzero} &= \Set{(z, \zeta) \in W_{\Zkzero} \mmid \abs{z} \leq 1, \abs{\zeta} \leq \exp\textstyle\frac{\partial s_{\Zkzero}}{\partial c_2}(q(z, \zeta))}, \\
  D_j &= \Set{(z, \zeta) \in W_j \mmid \abs{z} \leq 1, \abs{\zeta} \leq 1}, \qquad \text{for $j \in \Z_k \setminus \Set{\Zkzero}$}.
\end{align*}
Note that these spaces depend on $\delta$.

Define, for $j, \ell \in \Z_k$, diffeomorphisms $G_{j, \ell} = G_{\Zkzero, \ell} \circ G_{\Zkzero, j}^{-1} \colon U_j \to U_\ell$ and by \cref{lem:symplecto-auto,lem:symplecto-endo} symplectomorphisms $\varphi_{j, j+\Zkone} \colon W_{j, \mathrm{nu}} \to W_{j+\Zkone, \mathrm{ns}}$ as
\begin{align} \label{eq:def-varphi-j-jplusone}
  \varphi_{j, j+\Zkone} = \begin{cases}
    \varphi_{G_{j, j+\Zkone}} \circ \Psi_{-\kappa}, & j \neq \Zknegone; \\
    \Psi_{-\der s_{\Zkzero}} \circ \varphi_{G_{\Zknegone, \Zkzero}} \circ \Psi_{-\kappa}, & j = \Zknegone.
  \end{cases}
\end{align}
Let $\mathcal{G}$ be the groupoid generated by the restrictions of $\varphi_{j, j+\Zkone}$ for $j \in \Z_k$ onto open subsets.
Recall $\Gamma_k$ is the cycle graph with $k$ vertices.
Consider its fundamental groupoid $\Pi(\Gamma_k)$ whose elements are of the form $[j, \ell]_p$, where $j, \ell \in \Z_k, p \in \Z$ and $j + \overline{p} = \ell$,
The multiplication is given by concatenation $[\ell, j']_{p'} \cdot [j, \ell]_p = [j, j']_{p+p'}$.
Any element $[j, \ell]_p$ of $\Pi(\Gamma_k)$ corresponds to an element of $\mathcal{G}$ :
\begin{align*}
  \begin{cases}
    \varphi_{[j, j]_{\Zkzero}} = \identity \colon W_j \to W_j, & p = 0; \\
    \varphi_{[j, j+\Zkone]_1} = \varphi_{j, j+\Zkone} \colon W_{j, \mathrm{nu}} \to W_{j+\Zkone, \mathrm{ns}}, & p = 1; \\
    \varphi_{[j, j-\Zkone]_{-1}} = \varphi_{j-\Zkone, j}^{-1} \colon W_{j, \mathrm{ns}} \to W_{j-\Zkone, \mathrm{nu}}, & p = -1; \\
    \varphi_{[j, j+\overline{p}]_p} = \varphi_{j+\overline{p-1}, j+\overline{p}} \circ \dotsb \circ \varphi_{j+\Zkone, j+\overline{2}} \circ \varphi_{j, j+\Zkone} \colon W_{j, \mathrm{r}} \to W_{j+\overline{p}, \mathrm{r}}, & p \geq 2; \\
    \varphi_{[j, j+\overline{p}]_p} = \varphi_{j+\overline{p}, j+\overline{p+1}}^{-1} \circ \dotsb \circ \varphi_{j-\overline{2}, j-\Zkone}^{-1} \circ \varphi_{j-\Zkone, j}^{-1} \colon W_{j, \mathrm{r}} \to W_{j+\overline{p}, \mathrm{r}}, & p \leq -2.
  \end{cases}
\end{align*}
Actually, $\mathcal{G}$ consists of restrictions of $\varphi_{[j, \ell]_p}$ for all $[j, \ell]_p \in \Pi(\Gamma_k)$ restricted to open subsets, and $\mathcal{G}$ is a groupoid of symplectomorphisms.
Note that for any $[j, \ell]_p \in \Pi(\Gamma_k)$, we have
\begin{equation*}
  q \circ \varphi_{[j, \ell]_p} = G_{j, \ell} \circ q.
\end{equation*}
Define a smooth function
\begin{align*}
  f_L \colon kW_\mathrm{r} &\to \R, \\
  W_{j, \mathrm{r}} \ni (z, \zeta) &\mapsto \frac{\partial g_{\Zkzero, j}}{\partial c_2}(G_{\Zkzero, j}^{-1}(q(z, \zeta))) \ln \abs{z}
\end{align*}
for $j \in \Z_k$.

\begin{lemma} \label{lem:Psi-tau-discrete}
  For any $j \in \Z_k$ and $(z, \zeta) \in W_j \subset kW$ there is a $p \in \Z$ such that $\varphi_{[j, j+\overline{p}]_p}(z, \zeta) \in D_{j+\overline{p}} \subset kD$.
  For any $j \in \Z_k$, we have $f_L \circ \varphi_{[j, j+\overline{p}]_p} - f_L \to \infty$ uniformly as $p \to \infty$ and $f_L \circ \varphi_{[j, j+\overline{p}]_p} - f_L \to -\infty$ uniformly as $p \to -\infty$, both on $W_{j, \mathrm{r}}$.
\end{lemma}

\begin{proof}
  Define functions $L_j \colon U_{\Zkzero, \mathrm{r}} \to \R$, $j \in \Z_k$, as
  \begin{align*}
    L_{\Zkzero}(c) &= -\ln \abs c + \frac{\partial s_{\Zkzero}}{\partial c_2}(c) \geq (1 - \delta) \abs{\ln \delta} + (\mathsf{v}_{\Zkzero}^Y - \delta), \\
    L_j(c) &= -\frac{\partial g_{\Zkzero, j}}{\partial c_2}(c) \ln \abs{G_{\Zkzero, j}(c)} \geq (\mathsf{w}_{\Zkzero, j}^Y - \delta) \abs{\ln \delta}, \qquad \text{for $j \in \Z_k \setminus \Set{\Zkzero}$};
  \end{align*}
  the inequalities hold by \cref{eq:estimate-delta}.

  Recall the definition of the function $r$ in \cref{eq:def-r} and we have $f_L = \Pa{\frac{\partial g_{\Zkzero, j}}{\partial c_2} \circ G_{\Zkzero, j}^{-1} \circ q} \cdot r$ on $W_{j, \mathrm{r}}$.
  By \cref{eq:r-translation,eq:r-varphi_G,eq:Psi-kappa-ext,eq:def-varphi-j-jplusone} we have, for any $p \in \Z$,
  \begin{align*}
    r \circ \varphi_{[0, \overline{p}]_p} &= \frac{r \circ \varphi_{[0, \overline{p-1}]_{p-1}} - \ln \abs{G_{\Zkzero, \overline{p-1}}\circ q}}{\frac{\partial g_{\overline{p-1}, \overline{p}}}{\partial c_2} \circ q}, \qquad \text{for $\overline{p} \neq 0$}; \\ 
    r \circ \varphi_{[0, \overline{p}]_p} &= \frac{r - \ln \abs{G_{\Zkzero, \overline{-1}} \circ q}} {\frac{\partial g_{\Zknegone, \Zkzero}}{\partial c_2} \circ q} - \frac{\partial s_{\Zkzero}}{\partial c_2} \circ q, \qquad \text{for $\overline{p} = 0$}.
  \end{align*} 
  on $W_{\Zkzero, \mathrm{r}}$.
  Therefore, $f_L$ and $L_j$ are related by
  \begin{align*}
    &\phantom{={}} f_L \circ \varphi_{[0, \overline{p}]_p} - f_L \circ \varphi_{[0, \overline{p-1}]_{p-1}} \\
    &= \Pa{\frac{\partial g_{\Zkzero, \overline{p}}}{\partial c_2} \circ q} \cdot \Pa{r \circ \varphi_{[0, \overline{p}]_p}} - \Pa{\frac{\partial g_{\Zkzero, \overline{p-1}}}{\partial c_2} \circ q} \cdot \Pa{r \circ \varphi_{[0, \overline{p-1}]_{p-1}}} \\
    &= L_{\overline{p}} \circ q.
  \end{align*}
  on $W_{\Zkzero, \mathrm{r}}$, and then we have
  \begin{align*}
    f_L \circ \varphi_{[\Zkzero, \overline{p}]_p} &= \begin{cases}
      f_L + \sum_{s = 1}^{p} L_{\overline{s}} \circ q, & p \geq 0, \\
      f_L - \sum_{s = p+1}^{0} L_{\overline{s}} \circ q, & p < 0.
    \end{cases}
  \end{align*}
  Together with the fact that $L_j$, $j \in \Z_k$, are positive and bounded away from zero, we conculde that $f_L \circ \varphi_{[j, j+\overline{p}]_p} - f_L$ diverges to $\infty$ uniformly as $p \to \infty$ and diverges to $-\infty$ uniformly as $p \to -\infty$ for $j = \Zkzero$ and then for any $j \in \Z_k$.

  For any fixed $(z, \zeta) \in W_{\Zkzero}$, we aim to find a $p \in \Z$ such that $\varphi_{[\Zkzero, \overline{p}]_p}(z, \zeta) \in D_{\overline{p}}$.
  Let $c = q(z, \zeta)$.
  Suppose $(z, \zeta) \in W_{\Zkzero, \mathrm{r}}$.
  If $f_L(z, \zeta) \leq 0$, there is a $p \in \Z$, $p \geq 1$ such that
  \begin{align*}
    -\sum_{s = 1}^{p} L_{\overline{s}}(c) \leq f_L(z, \zeta)
     \leq -\sum_{s = 1}^{p-1} L_{\overline{s}}(c)
  \end{align*}
  so $-L_{\overline{p}}(c) \leq f_L \circ \varphi_{[\Zkzero, \overline{p}]_p}(z, \zeta) \leq 0$ and $\varphi_{[\Zkzero, \overline{p}]_p}(z, \zeta) \in D_{\overline{p}}$; if otherwise $f_L(z, \zeta) > 0$, there is a $p \in \Z$, $p \leq 0$ such that $\varphi_{[\Zkzero, \overline{p}]_p}(z, \zeta) \in D_{\overline{p}}$ by a similar argument.
  Suppose otherwise $(z, \zeta) \in W_{\Zkzero} \setminus W_{\Zkzero, \mathrm{r}}$.
  If $\zeta = 0$ and $\abs{z} \leq 1$, or $z = 0$ and $\abs{\zeta} \leq 1$, we already have $(z, \zeta) \in D_{\Zkzero}$.
  If $\zeta = 0$ and $\abs{z} > 1$, then $\varphi_{[\Zkzero, \Zknegone]_{-1}}(z, \zeta) = (0, \zeta') \in D_{\Zknegone}$ since $0 < \abs{\zeta'} < 1$.
  If $z = 0$ and $\abs{\zeta} > 1$, then $\varphi_{[\Zkzero, \Zkone]_1}(z, \zeta) = (z', 0) \in D_{\Zkone}$ since $0 < \abs{z'} < 1$.
  Analogously, for any $(z, \zeta) \in W_j$, $j \in \Z_k$, there is a $p \in \Z$ such that $\varphi_{[j, j+\overline{p}]_p}(z, \zeta)$ is in $D_{j+\overline{p}}$.
\end{proof}

We define an equivalence equation $\sim_{\mathcal{G}}$ on $kW$ as $x \sim_{\mathcal{G}} y$ if and only if there is a $\varphi \in \mathcal{G}$ such that $y = \varphi(x)$.
Let $M = kW / \sim_{\mathcal{G}}$ be the quotient space, $\lambda \colon kW \to M$, $\lambda_j \colon W_j \to M$, $j \in \Z_k$ be the quotient maps.
Let $\Delta_{\mathcal{G}} = \Set{(x, y) \in kW \times kW \mmid x \sim_{\mathcal{G}} y}$.

\begin{lemma} \label{lem:symplectic-manifold-structure}
  The topological space $M$ can be uniquely realized as a symplectic manifold with a symplectic structure $\omega$ and a smooth function $F \colon M \to \R^2$ such that for every $j \in \Z_k$ the map $\lambda_j \colon (W_j, \omega_0) \to (\lambda_j(W_j), \omega)$ is a symplectomorphism and $G_{\Zkzero, j} \circ F \circ \lambda_j = \Res{q}_{W_j}$.
\end{lemma}

\begin{proof}
  We want to prove that $M$ is a topological manifold with the quotient topology.

  \emph{The map $\lambda_j$ is open:}
  for any open set $V \subset W_j$, the preimage
  \begin{equation*}
    \lambda_j^{-1}(\lambda_j(V)) = V \cup \varphi_{[j, j+\Zkone]_1}(V \cap W_{j, \mathrm{nu}}) \cup \varphi_{[j, j-\Zkone]_{-1}}(V \cap W_{j, \mathrm{ns}}) \cup \bigcup_{p \in \Z, \abs{p} \geq 2} \varphi_{[j, j+\overline{p}]_p}(V \cap W_{j, \mathrm{r}})
  \end{equation*}
  is open, so $\lambda_j$ is an open map.

  \emph{The map $\lambda_j$ is locally injective:}
  we need to prove that, any $x \in W_j$ has a neighborhood $V$ in $W_j$ such that for any $p \in \Z \setminus \Set{0}$, as long as $x$ is in the domain, the map $\varphi_{[j, j]_{pk}}$ sends $x$ outside of $V$.
  If $k \geq 2$, then $x \in W_{j, \mathrm{r}}$.
  This is a consequence of \cref{lem:Psi-tau-discrete}.
  If $k = 1$ and $x \in W_{\Zkzero, \mathrm{s}} \setminus \Set{0}$, we have $\varphi_{[\Zkzero, \Zkzero]_1}(x) \in W_{\Zkzero, \mathrm{u}}$ away from $x$.
  The case $k = 1$ and $x \in W_{\Zkzero, \mathrm{u}} \setminus \Set{0}$ is analogous.

  \emph{The subset $\Delta_{\mathcal{G}}$ is closed in $kW \times kW$:}
  suppose there is a sequence of points $((x_i, y_i))_{i = 1}^\infty \subset \Delta_{\mathcal{G}}$ converging to $(x_\infty, y_\infty) \in kW \times kW$.
  Assume, without loss of generality, that $(x_\infty, y_\infty) \in W_{\Zkzero} \times W_j$ for some fixed $j \in \Z_k$.
  Since $W_{\Zkzero}, W_j$ are open in $kW$, we can assume $((x_i, y_i))_{i = 1}^\infty \subset W_{\Zkzero} \times W_j$.
  There is $[0, \overline{p_i}]_{p_i} \in \Pi(\Gamma_k)$ such that $y_i = \varphi_{[0, \overline{p_i}]_{p_i}}(x_i)$.
  If there is a subsequence $\Set{p_{i_m}}$ of $p_i$ with $p_{i_m} = p_0 \in \Z$, then $y_{i_m} = \varphi_{[0, \overline{p_0}]_{p_0}}(x_{i_m})$.
  In this case, $y_\infty = \varphi_{[0, \overline{p_0}]_{p_0}}(x_\infty)$, so $(x_\infty, y_\infty) \in \Delta_{\mathcal{G}}$.
  Otherwise, by descending to a subsequence we can assume $\abs{p_i} \to \infty$, so for $i$ large, $x_i \in W_{\Zkzero, \mathrm{r}}$ and $y_i \in W_{j, \mathrm{r}}$.
  By \cref{lem:Psi-tau-discrete}, we have $\abs{f_L(x_i) - f_L(y_i)} \to \infty$, which contradicts $(x_i, y_i) \to (x_\infty, y_\infty)$.

  Since $\lambda_j$, $j \in \Z_k$ are open and locally injective, $\lambda_j$, $j \in \Z_k$ are local homeomorphisms, and $M$ is locally Euclidean.
  Since $\lambda_j$, $j \in \Z_k$ are open and $\Delta_{\mathcal{G}} \subset kW \times kW$ is closed, $M$ is Hausdorff.
  Since $W_j$, $j \in \Z_k$ are second countable, $M = \bigcup_{j \in \Z_k} \lambda_j(W_j)$ is second countable.
  We conclude that $M$ is a topological manifold.

  Noting that the maps $\varphi_{j, \ell}$, $j, \ell \in \Z_k$ are symplectomorphisms satisfying $q \circ \varphi_{j, \ell} = G_{j, \ell} \circ q$, there is a unique symplectic structure $\omega$ on $M$, and a smooth function $F \colon M \to \R^2$ such that $\lambda_j^* \omega = \omega_0$ and $G_{\Zkzero, j} \circ F \circ \lambda_j = \Res{q}_{W_j}$.
\end{proof}

\begin{lemma} \label{lem:invariant-surjectivity}
  The map $\Phi \colon \Mffobk{k} \to \Iffobk{k}$ is surjective.
\end{lemma}

\begin{proof}
  Let $m_j = \lambda_j(0)$ and $\mu_j = \Res{\lambda_j}_{\lambda_j(W_j)}^{-1} \colon (\lambda_j(W_j), \omega) \to (W_j, \omega_0)$ be a symplectomorphism for $j \in \Z_k$.
  Finally, we need to show that, the construction $(M, \omega, F)$ in \cref{lem:symplectic-manifold-structure} lies inside of $\IntSysFFk{k}$, has a singularity atlas $\Pa{(\mu_j, G_{\Zkzero, j})}_{j \in \Z_k}$ for singular points $m_j$, $j \in \Z_k$, compatible with some $(\alpha_1, \alpha_2) \in \Dir(M, \omega, F)$ such that \eqref{eq:invariants-coincide} holds.

  \emph{The triple $(M, \omega, F)$ is in $\IntSysFFk{k}$:}
  The triple $(M, \omega, F)$ is an integrable system since it is locally an integrable system everywhere, and the only singular points of $F$ are $m_j$ on $\mathcal{F}$, $j \in \Z_k$, which are of focus-focus type.
  To show that $F$ is proper, let $K \subset U_0$ be any compact subset.
  By \cref{lem:Psi-tau-discrete}, $\lambda(kW) = \lambda(kD)$.
  Since $q^{-1}(G_{\Zkzero, j}(K)) \cap D_j$ is compact, $F^{-1}(K) = \bigcup_{j \in \Z_k} \lambda_j(q^{-1}(G_{\Zkzero, j}(K)) \cap D_j)$ is compact.
  The fibers of $F$ are connected since $q$ has connected fibers.

  \emph{Computation of $\Period^{(M, \omega, F)}$:}
  Let $U \subset U_{\Zkzero, \mathrm{r}}$ be a simply connected open set.
  Note that $\Res{\kappa}_U \in (\Omega^1 / 2\pi \Period)(\R^2_{\mathrm{r}})$ and let $\kappa_U \in \Omega^1(U)$ be a representative of $\Res{\kappa}_U$.
  Let $\Res{\alpha_2}_U = \der s_{\Zkzero} - \sum_{j \in \Z_k} G_{\Zkzero, j}^* \kappa_U \in Z^1(U)$.
  We have, in $F^{-1}(U)$,
  \begin{align*}
    \Res{\varphi_{[\Zkzero, \Zkzero]_k}}_{F^{-1}(U)} &= \Psi_{-\der s_{\Zkzero}} \circ \varphi_{G_{\Zknegone, \Zkzero}} \circ \Psi_{-\kappa_U} \circ \cdots \circ \Psi_{-\kappa_U} \circ \varphi_{G_{\Zkone, \overline{2}}} \circ \Psi_{-\kappa_U} \circ \varphi_{G_{\Zkzero, \Zkone}} \circ \Psi_{-\kappa_U} \\
    &= \Psi_{-\der s_{\Zkzero} - \sum_{j \in \Z_k} G_{\Zkzero, j}^* \kappa_U} \circ \varphi_{G_{\Zknegone, \Zkzero}} \circ \cdots \circ \varphi_{G_{\Zkone, \overline{2}}} \circ \varphi_{G_{\Zkzero, \Zkone}} \\
    &= \Psi_{-\Res{2\pi \alpha_2}_U}.
  \end{align*}
  So $\Res{\alpha_2}_U \in \Period^{(M, \omega, F)}(U)$.
  Let $\alpha_1 = dc_1 \in \Omega^1(U_{\Zkzero})$, then $\Res{\alpha_1}_U \in \Period^{(M, \omega, F)}(U)$.
  On the other hand, for any $\tau \in Z^1(U)$ to be a period form, it has to satisfy $\Psi_{2\pi \tau} = \varphi_{[\Zkzero, \Zkzero]_{pk}}$ for some $p \in \Z$.
  Therefore, $\Period^{(M, \omega, F)}(U)$ is the abelian group generated by $\Res{\alpha_1}_U, \Res{\alpha_2}_U$.
  Similarly, we have $\Period^{(M, \omega, F)}(U) = \alpha_1 \Z$ if $U$ is an open neighborhood of $0$.

  \emph{Computation of the invariants:}
  For each $j \in \Z_k$, $(\mu_j, G_{\Zkzero, j})$ is an Eliasson local chart near $m_j$ since $q \circ \mu_j = G_{\Zkzero, j} \circ F$.
  For $j = \Zkzero$, note that $\alpha_1 = \der c_1$ and $\frac{\partial}{\partial c_2} \intprod \alpha_2 = L$, so $(\mu_{\Zkzero}, \identity)$ is compatible with $(\alpha_1, \alpha_2)$.
  For $j \in \Z_k$, since $\der G_{\Zkzero, j}$ has positive diagonal entries near the origin, $(\mu_j, G_{\Zkzero, j})$ is compatible with $(\alpha_1, \alpha_2)$.
  By the construction of $M$, any flow line of $\mathcal{X}_{\alpha_2}$ with $\upalpha$-limit $m_j$ for some $j \in \Z_k$ has $\upomega$-limit $m_{j+\Zkone}$, so $\Pa{(\mu_j, G_{\Zkzero, j})}_{j \in \Z_k}$ is a singularity atlas compatible with $(\alpha_1, \alpha_2)$.

  Now since $\der s_{\Zkzero} = 2\pi \alpha_2 - \sum_{j \in \Z_k} G_{\Zkzero, j}^* \kappa$ and $s_{\Zkzero}(0) = 0$, the action Taylor series $\mathsf{s}_{\Zkzero}$ at $m_{\Zkzero}$ is, $\mathsf{s}_{\Zkzero} = \Tl_0[s_{\Zkzero}] = \mathsf{v}_{\Zkzero}$.
  The transition cocycle $\Pa{g_{j, \ell}}_{j, \ell \in \Z_k}$ is such that $\mathsf{g}_{j, \ell} = \Tl_0[\proj_2 \circ G_{j, \ell}] = \mathsf{w}_{j, \ell}$ for $j, \ell \in \Z_k$.
\end{proof}

\section{Proof of the main theorem} \label{sec:proof-main-theorem}

\cref{thm:invariant-bijection} follows from \cref{lem:invariant-uniqueness,lem:invariant-injectivity,lem:invariant-surjectivity} put together.

\begin{proof}[Proof of \cref{thm:invariant-bijection}]
  The map
  \begin{align*}
    \Phi \colon \Iffobk{k} &\to \mathsf{R}_{2\pi X} \times \mathsf{R}_+^{k-1} \\
    \Pa{\mathsf{s}_{\Zkzero}, \dotsc, \mathsf{s}_{\Zknegone}, \mathsf{g}_{\Zkzero, \Zkzero}, \dotsc, \mathsf{g}_{\Zkzero, \Zknegone}, \dotsc, \mathsf{g}_{\Zknegone, \Zkzero}, \dotsc, \mathsf{g}_{\Zknegone, \Zknegone}} &\mapsto \Pa{\mathsf{s}_{\Zkzero}, \mathsf{g}_{\Zkzero, \Zkone}, \mathsf{g}_{\Zkone, \overline{2}}, \dotsc, \mathsf{g}_{\Zknegone, \Zkzero}}
  \end{align*}
  is a bijection.
\end{proof}

\appendix

\section{Technical results on flat functions} \label{sec:flat-function}

Recall that $\Tl_0[f] \in \R[[T^*_0U]]$ denotes the Taylor series of $f \colon U \to \R$ at the origin $U$for $U$ an open neighbourhood of $0$ in $\R^2$.
Note how the Taylor series, as a formal power series, depends on the choice of a basis of $T^*_0U$.
Note that the Taylor series only depends on the germ of the function.

\begin{definition} \label{def:flat-function}
  We call $f$ a \emph{flat function} at $0$, or the function $f$ is flat at $0$, if $\Tl_0b[f] = 0$.
  Denote by $\inftes$ the space of flat functions at $0$.
\end{definition}

Note that, by the Fa\`a di Bruno's formula, the Taylor series of the composition of smooth maps is the composition of their Taylor series.
This is why the definition of flat functions is independent of the choice of the basis of $T^*_0U$.

We will use the multi-index notations in \cref{lem:logarithm-Ck,lem:logarithm-flat}.
A multi-index $j$ is a pair $(j_1, j_2)$ where $j_1, j_2 \in \Z_{\geq 0}$.
We use $\abs{j} = j_1 + j_2$.
If $c = (c_1, c_2) \in \R^2$ then $c^j = c_1^{j_1} c_2^{j_2}$.
If $f$ is a function on an open subset of $\R^2$ then $\partial^j f = \frac{\partial^{\abs{j}} f}{\partial c_1^{j_1} \partial c_2^{j_2}}$.

\begin{lemma} \label{lem:logarithm-Ck}
  Let $m \in \Z_{\geq 0}$.
  Let $g_j \colon \R^2 \to \R$ be a smooth function for multi-index $j$ with $\abs{j} = m$ and let $g(c) = \sum_{\abs{j} = m} g_j(c) c^j$.
  Then the function $g \ln \abs \cdot \colon \R^2_\mathrm{r} \to \R$ can be extended to a $\mathrm{C}^{m-1}$-function on $\R^2$ if $m \geq 1$.
  Furthermore, the extension is $\mathrm{C}^m$ if and only if $g_j(0) = 0$ for all $j$.
\end{lemma}

\begin{proof}
  Let $L_m$, $m \in \Z_{\geq 0}$, be the $\mathrm{C}^\infty(\R^2)$-vector space spanned by functions of the form $\Pa{\R^2_{\mathrm{r}} \ni c \mapsto c^j \ln \abs{c}}$ for which $j$ is a multi-index with $\abs{j} \geq m$, extended onto $\R^2$ by zero.
  Let $Q_m$, $m \in \Z$, be the $\mathrm{C}^\infty(\R^2)$-vector space spanned by functions of the form $\Pa{\R^2_{\mathrm{r}} \ni c \mapsto \frac{c^j}{\abs{c}^{m_0}}}$ for which $j$ is a multi-index, $m_0 \in \Z_{\geq 0}$, and $\abs{j} \geq m_0 + m$, extended onto $\R^2$ by zero.
  By direct calculations,
  \begin{align*}
    \frac{\partial}{\partial c_1} \Pa{c^j \ln \abs{c}} &= j_1 c_1^{j_1-1} c_2^{j_2} \ln \abs{c} + \frac{c_1^{j_1+1} c_2^{j_2}}{\abs{c}^2}, \\
    \frac{\partial}{\partial c_1} \Pa{\frac{c^j}{\abs{c}^{m_0}}} &= j_1 \frac{c_1^{j_1-1} c_2^{j_2}}{\abs{c}^{m_0}} - m_0 \frac{c_1^{j_1+1} c_2^{j_2}}{\abs{c}^{m_0+2}},
  \end{align*}
  which implies that if $f \in L_m$ then $\frac{\partial}{\partial c_1} f, \frac{\partial}{\partial c_2} f \in L_{m-1} + Q_{m-1}$ for $m \geq 1$, and if $f \in Q_m$ then $\frac{\partial}{\partial c_1} f, \frac{\partial}{\partial c_2} f \in Q_{m-1}$ for $m \in \Z$.
  Therefore, $Q_0 \subseteq \mathrm{B}_0$ the set of functions on $\R^2$ bounded in a neighborhood of the origin, $L_1, Q_1 \subseteq \mathrm{C}^0$ and then $L_m, Q_m \subseteq \mathrm{C}^{m-1}$ for $m \geq 1$.

  We aim to find $L_0 \cap \mathrm{B}_0$ and let $f = h \ln \abs \cdot \in L_0$ for some $h \in \mathrm{C}^\infty$.
  On one hand, if $f \in L_0 \cap \mathrm{B}_0$, then the boundedness requires that $h(0) = 0$.
  On the other hand, if $h(0) = 0$, then $h(c) = h_1 c_1 + h_2 c_2$ for some $h_1, h_2 \in \mathrm{C}^\infty$ and then $f \in L_1 \subseteq \mathrm{C}^0$.
  Therefore, $L_0 \cap \mathrm{B}_0 = L_1$.

  In particular, given $g(c) = \sum_{\abs{j} = m} g_j(c) c^j$, we have $g \ln \abs \cdot \in L_m \subseteq \mathrm{C}^{m-1}$ if $m \geq 1$, and for any multi-index $j$ with $\abs{j} = m$, we have
  \begin{align*}
    \partial^j \Pa{g(c) \ln \abs{c}} \in j!\, g_j(c) \ln \abs c + L_1 + Q_0.
  \end{align*}
  On one hand, $g \ln \abs \cdot \in \mathrm{C}^m$ requires that $\partial^j \Pa{g(c) \ln \abs{c}} \in \mathrm{B}_0$, and hence $g_j(0) = 0$.
  On the other hand, $g_j(0) = 0$ for every multi-index $j$ with $\abs{j} = m$ implies that $g \ln \abs \cdot \in L_{m+1} \subseteq \mathrm{C}^m$.
\end{proof}

\begin{lemma} \label{lem:logarithm-flat}
  For a smooth function $f$ on $\R^2$, the function $f\ln\abs\cdot$ on $\R^2_\mathrm{r}$ can be smoothly extended onto $\R^2$ only when $f$ is flat.
  If $f$ is flat, the extension of $f\ln\abs\cdot$ is also flat.
\end{lemma}

\begin{proof}
  Using Taylor expansion of $f$, for any $m \in \N$, there exist smooth functions $g_j \colon \R^2 \to \R$ for all multi-indices $j$ with $\abs{j} = m+1$ such that
  \begin{equation} \label{eq:taylor-expansion}
    f(c) = \sum_{\abs{j} = 0}^m \frac{1}{j!} \partial^j f(0) c^j + \sum_{\abs{j} = m+1} \frac{1}{j!} g_j(c) c^j.
  \end{equation}
  By \cref{lem:logarithm-Ck,eq:taylor-expansion}, $f\ln\abs\cdot \in \mathrm{C}^m$ if and only if $\partial^j f(0) = 0$ for any multi-index $j$ with $\abs{j} \leq m$.
  Therefore, $f\ln\abs\cdot \in \mathrm{C}^\infty$ if and only if $f \in \inftes$.

  Note that $\ln \abs c \in \mathcal{O}(\abs{c}^{-1})$.
  If $f \in \inftes$, then for any $m \in \N$, there exist smooth functions $g_j \colon \R^2 \to \R$ for all multi-index $j$ with $\abs{j} = m+1$ such that
  \begin{equation*}
    f(c) \ln \abs c = \sum_{\abs{j} = m+1} \frac{1}{j!} g_j(c) c^j \ln \abs c \in \mathcal{O}(\abs{c}^{m}).
  \end{equation*}
  Hence $f\ln\abs\cdot \in \inftes$.
\end{proof}

\begin{lemma} \label{lem:flat-times-tempered}
  If $g$ is a flat function on $\R^2$ and $h$ is a smooth function on $\R^2_\mathrm{r}$ that satisfies
  \begin{equation} \label{eq:tempered}
    \forall\text{multi-index}~j~\exists m_j \in \Z\text{ such that } \lim_{c \to 0} \abs{c}^{m_j} \abs{\partial^j h(c)} = 0.
  \end{equation}
  Then $f = gh$ on $\R^2_\mathrm{r}$ has a smooth extension $\tilde{f}$ on $\R^2$.
\end{lemma}

\begin{proof}
  We calculate the partial derivatives of $f$ for a multi-index $j$ and $m \in \Z$:
  \begin{equation} \label{eq:limit-of-derivative}
    \abs{c}^m \abs{\partial^j f(c)} \leq \sum_{0 \leq \ell \leq j} \binom{j}{\ell} \abs{c}^{m-m_\ell} \abs{\partial^{j-\ell} g(c)} \cdot \abs{c}^{m_\ell} \abs{\partial^\ell h(c)} \to 0
  \end{equation}
  as $c \to 0$.
  Here we use the fact that $\partial^{j-\ell} g$ is a flat function so it is dominated by any power of $\abs{c}$.
  Now let $\tilde{f}$ be the extension of $f$ by $\tilde{f}(0) = 0$.
  Then by \cref{eq:limit-of-derivative}, $\frac{\partial \tilde{f}}{\partial c_1}(0) = \lim_{\delta \to 0} \frac{f(\delta, 0)}{\delta} = 0$ and $\frac{\partial \tilde{f}}{\partial c_2}(0) = \lim_{\delta \to 0} \frac{f(0, \delta)}{\delta} = 0$, and then $\tilde{f} \in \mathrm{C}^1$.
  Inductively, we can show that $\tilde{f} \in \mathrm{C}^\infty$ and any higher order derivative of $\tilde{f}$ vanishes at the origin.
\end{proof}

\section*{Acknowledgments}

The first author was supported by NSF CAREER Grant DMS-1518420 and by a BBVA Foundation Grant for Scientific Research Projects with project title ``From Integrability to Randomness in Symplectic and Quantum Geometry''.
The second author was supported by NSF CAREER Grant DMS-1518420 and by Beijing Institute of Technology Research Fund Program for Young Scholars.
The authors thank San V{\~u} Ng{\d{o}}c for many helpful discussions on the topic of the paper and Joseph Palmer for helpful comments on a preliminary version of the paper.
The second author would also like to thank San V{\~u} Ng{\d{o}}c for the invitation to visit the Universit\'e de Rennes 1 in November and December of 2016.

Part of this paper was completed while both authors were at the University of California, San Diego and during the second author's visit to Tsinghua University in 2021.

\bibliographystyle{hplain}
\bibliography{ref}

\authaddresses

\end{document}